\documentclass[11pt,reqno,tbtags]{amsart}
\usepackage{amsmath}
\usepackage{mathrsfs}
\usepackage{amssymb}
\usepackage{amsfonts}
\usepackage{stmaryrd}
\usepackage{paralist}
\usepackage{graphics}
\usepackage[pdftex]{graphicx}
\usepackage[colorlinks=true, pdfstartview=FitV, linkcolor=blue, citecolor=blue, urlcolor=blue,pagebackref=false]{hyperref}
\hypersetup{urlcolor=blue, citecolor=red}
\usepackage{enumerate,graphicx}
\usepackage{bbding}
\usepackage{pifont}
\usepackage{mathrsfs}
\usepackage[usenames]{color}

\numberwithin{equation}{section}

\newtheorem{theorem}{Theorem}[section]
\newtheorem{corollary}{Corollary}

\newtheorem{lemma}[theorem]{Lemma}
\newtheorem{proposition}[theorem]{Proposition}

\theoremstyle{definition}
\newtheorem{definition}[theorem]{Definition}
\newtheorem{remark}[theorem]{Remark}

\newcommand{\eps}{\varepsilon}

\newcommand{\R}{\mathbb R}

\newcommand{\C}{\mathbb C}
\newcommand{\mO}{\mathcal{O}}

\newcommand{\tl}{\tilde}
\newcommand{\omg}{\omega}

\title[Approximate electromagnetic cloaking]
      {On Approximate electromagnetic cloaking by transformation media}

\begin{document}
\maketitle

\centerline{\scshape Hongyu Liu}
\medskip
{\footnotesize
 \centerline{ Department of Mathematics}
   \centerline{University of California, Irvine}
  \centerline{Irvine, CA 92697, USA}
  \centerline{Email:\ {\tt hongyul1@uci.edu}}
}

\medskip

\centerline{\scshape Ting Zhou }
\medskip
{\footnotesize
 \centerline{Department of Mathematics}
  \centerline{University of Washington}
   \centerline{ Seattle, WA 98195, USA}
   \centerline{Email:\ {\tt tzhou@uw.edu}}
}

\begin{abstract}
We give a comprehensive study on regularized approximate
electromagnetic cloaking via the transformation optics approach. The
following aspects are investigated:~(i)~near-invisibility cloaking
of passive media as well as active/radiating sources; (ii)~the
existence of cloak-busting inclusions without lossy medium lining;
(iii)~overcoming the cloaking-busts by employing a lossy layer
outside the cloaked region; (iv)~the frequency dependence of the
cloaking performances. We address these issues and connect the
obtained asymptotic results to singular ideal cloaking. Numerical
verifications and demonstrations are provided to show the sharpness
of our analytical study.
\end{abstract}

\section{Introduction}

A region is said to be \emph{cloaked} if its contents together with
the cloak are invisible to certain noninvasive wave detections.
Blueprints for making objects invisible to electromagnetic waves
were proposed by Pendry {\it et al.} \cite{PenSchSmi} and Leonhardt
\cite{Leo} in 2006. In the case of electrostatics, the same idea was
discussed by Greenleaf {\it et al.} \cite{GLU} in 2003. The key
ingredient for those constructions is that optical parameters have
transformation properties and could be {\it push-forwarded} to form
new material parameters. The obtained materials/media are called
{\it transformation media}, which we shall further examine in the
current work for cloaking of the full system of Maxwell's equations.

The transformation-optics-approach-based scheme proposed in
\cite{GLU,PenSchSmi} is rather singular. This poses much challenge
to both theoretical analysis and practical fabrication. In order to
avoid the singular structures, several regularized approximate
cloaking schemes are proposed in \cite{GKLU_2,KOVW,KSVW,Liu,RYNQ}.
The works \cite{GKLU_2} and \cite{RYNQ} are based on truncation,
whereas in \cite{KOVW,KSVW,Liu}, the `blow-up-a-point'
transformation in \cite{GLU,PenSchSmi} is regularized to be the
`blow-up-a-small-ball' transformation. The performances of both
regularization schemes have been assessed for cloaking of acoustic
waves to give successful near-invisibility effects. Particularly, in
\cite{KOVW}, the authors show that in order to `nearly-cloak' an
{\it arbitrary} content, it is necessary to employ an absorbing
(`lossy') layer lining right outside the cloaked region. Since
otherwise, there exist cloaking-busting inclusions which defy any
attempts of cloaking. This idea of introducing a lossy layer has
recently been intensively investigated for approximate acoustic
cloaking (see \cite{Ngu,NgVo}), whose behaviors are now much well
understood. However, very little progress has been made in the study of
approximate EM cloaking for full Maxwell's equations due to the more
complicated structure of Maxwell's equations. This is the main concern
of the present article.

We have considered both the `truncation-scheme' and the
`blow-up-a-small-ball-scheme' for approximate EM cloaking. However,
in our study, the two regularization schemes have the same
performances for near-invisibility, so we focus our exposition on
the latter one. Based on a model problem, the following aspects on
the approximate EM cloaking are addressed in detail.

(i)~The near-cloak of EM waves for both passive media and
active/radiating sources. For approximate cloaking of passive media,
the near-cloak is shown to be within $\rho^3$ of the singular ideal
cloaking, where $\rho$ is the regularization parameter. Whereas if
there is a delta point source present in the cloaked region, the
near-cloak is shown to be within $\rho^2$ of the perfect cloaking.
That is, we could still achieve the near-invisibility effect, but
with one order reduction on the approximation. Compared to the
near-invisibility assessments in
\cite{GKLU_2,KOVW,Liu,Ngu,NgVo,RYNQ} for approximate acoustic
cloaking (which is of $\mathcal{O}(\rho)$ when spatial dimension is
3; and $\mathcal{O}(1/\ln\rho)$ when spatial dimension is 2), the
performances for near-cloak of EM waves are much better. We point
out that the study in \cite{KOVW,Liu,Ngu,NgVo,RYNQ} lacks the
analysis on the approximate cloaking when there is an active/radiating
source present inside the cloaked region. Another rather
interesting observation we would like to make is that in
\cite{GKLU}, it is shown one cannot perfectly cloak an
$H^{-1}$-source inside the cloaked region since otherwise there
would be a conflict with certain `hidden' boundary conditions of
the finite energy solutions underlying the singular ideal EM
cloaking, but our analysis here shows that one could nearly-cloak a
delta point source inside the cloaked region.

(ii)~If one allows that the contents in the cloaked region could be
{\it arbitrary}, then for a fixed near-cloak construction, there
always exist cloaking-busting inclusions which defy any attempts of
cloaking. These are similar to the resonant inclusions observed in
\cite{KOVW} for approximate acoustic cloaking. Following
\cite{KOVW}, we employ a lossy layer with conducting medium outside
the cloaked region to overcome the resonance and re-achieve all the
approximate cloaking results for passive media and active sources in
(i).

(iii)~The performance of the approximate EM cloaking in
asymptotically low and high frequency regimes. We show that it is
impossible, with a fixed near-cloaking scheme, to obtain cloaking
uniformly in frequency, especially for the cloaking of
active/radiating objects. Our observation is closely related to the
very recent study in \cite{NgVo}, where frequency dependence for the
approximate acoustic cloaking is considered.

(iv)~The limiting behaviors of solutions to regularized approximate
cloaking problems, and their connections to finite energy solutions
considered in \cite{GKLU} for singular ideal cloaking problems.

Our study has been mainly restricted to spherical cloaking devices
with uniform cloaked contents. We base our analysis on spherical
wave functions expansions of EM wave fields. Nonetheless, we believe
similar results would equally hold for general approximate EM
cloaking study.

In this paper, we focus entirely on transformation-optics-approach
in constructing cloaking devices. We refer to \cite{GKLU5,Nor,YYQ}
for state-of-the-art surveys on the rapidly growing literature and
many striking applications of transformation optics. But we would
also like to mention in passing the other promising cloaking schemes
including the one based on anomalous localized resonance \cite{MN},
and another one based on special (object-dependent) coatings
\cite{AE}.

The rest of the paper is organized as follows. In Section 2, we
present the basics on transformation optics in a rather general
setting and apply them to the construction of EM cloaking devices.
Sections 3--5 are devoted to the main results, respectively on,
cloaking of passive media, cloaking of radiating objects and,
cloaking-busting inclusions and retaining of cloaking by employing a
lossy layer. The numerical experiments are given in Section 6.

\section{Transformation optics and electromagnetic
cloaking}\label{sect:2}

Let $\Omega$ be a bounded body in $\mathbb{R}^3$ whose electric
permittivity, conductivity, and magnetic permeability are described
by the $\mathbb{R}^{3\times 3}$-valued functions
$\varepsilon,\sigma$ and $\mu$, respectively. Consider the
time-harmonic electric field $E$ and magnetic field $H$ inside
$\Omega$ satisfying Maxwell's equations
\begin{equation}\label{eq:Maxwell eqn}
\nabla\times E=i\omega\mu H,\quad \nabla\times
H=-i\omega(\varepsilon+i\frac{\sigma}{\omega})E+J\qquad\mbox{in\ \
$\Omega$}
\end{equation}
with $\omega>0$ representing a frequency, $J$ an internal current
density. Let $\nu$ be the exterior unit normal on the boundary
$\partial\Omega$. By $\Lambda_{\varepsilon,\mu,\sigma,J}^\omega$ we
denote the linear mapping that takes the tangential component of
$E|_{\partial\Omega}$ to that of $H|_{\partial\Omega}$, i.e.,
\begin{equation}
\Lambda_{\varepsilon,\mu,\sigma,J}^\omega(\nu\times
E|_{\partial\Omega})=\nu\times H|_{\partial\Omega}.
\end{equation}
$\Lambda_{\varepsilon,\mu,\sigma,J}^\omega$ is known as {\it
impedance map} which encodes the exterior (boundary) measurements of
the EM wave fields. In noninvasive detections, one intends to
recover the interior object, namely $\mu,\varepsilon, \sigma$ and
$J$, by knowing $\Lambda^\omega_{\varepsilon,\mu,\sigma,J}$. It is
pointed out that knowledge of the impedance map is equivalent to
that of the corresponding scattering measurements (cf. \cite{CK}).
We refer readers to \cite{OPS1} and \cite{OS} for uniqueness results
of this inverse problem.
Throughout the rest of the paper, we shall denote by
$\{\Omega;\varepsilon,\mu,\sigma, J\}$ the object (EM medium and
internal current) supported in $\Omega$. We would also use
$\Lambda_0^\omega$ to denote the impedance map in the free space;
that is, it corresponds to the case with $\varepsilon=\mu=I$,
$\sigma=0$ and $J=0$ in $\Omega$. In this context, {\it
invisibility cloaking} can be generally introduced as follows.

\begin{definition}\label{def:1}
Let $\Omega$ and $D$ be bounded domains in $\mathbb{R}^3$ with
$D\Subset\Omega$. $\Omega\backslash\bar{D}$ and $D$ represent,
respectively, the cloaking region and the cloaked region.
$\{\Omega\backslash\bar{D}; \varepsilon_c,\mu_c, \sigma_c\}$ is said
to be an {\it invisibility cloaking} for the region $D$ if
\begin{equation}
\Lambda_{\varepsilon_e,\mu_e,\sigma_e,J_e}^\omega=\Lambda_{0}^\omega
\ \ \ \mbox{on $\partial\Omega$ for all $\omega>0$},
\end{equation}
where the extended object $\{\Omega; \varepsilon_e,\mu_e,\sigma_e,
J_e\}$ is given by
\begin{equation}\label{eq:extended object}
\{\Omega;\varepsilon_e,\mu_e,\sigma_e,J_e\}=\begin{cases} &
\{\Omega\backslash\bar{D}; \varepsilon_c,\mu_c,\sigma_c, 0\}\ \
\mbox{in\ $\Omega\backslash\bar{D}$},\\
& \{D;\varepsilon_a,\mu_a,\sigma_a,J_a\}\ \ \ \, \mbox{in\ $D$},
\end{cases}
\end{equation}
with $\{D;\varepsilon_a,\mu_a,\sigma_a,J_a\}$ being the target
object (which could be {\it arbitrary}).
\end{definition}
According to Definition~\ref{def:1}, the cloaking medium
$\{\Omega\backslash\bar{D};\varepsilon_c,\mu_c,\sigma_c\}$ makes the
target object, namely the interior EM medium
$\{D;\varepsilon_a,\mu_a,\sigma_a\}$ and the interior source/sink
$J$, invisible to exterior boundary measurements.

Next we present the transformation invariance of Maxwell's equations
and transformation properties of EM material parameters, which shall
form the key ingredients for our construction of invisibility
cloaking devices. To that end, we first briefly discuss the
well-posedness of the Maxwell equations (\ref{eq:Maxwell eqn}). In
the following, let $\Omega$ be an open bounded domain in
$\mathbb{R}^3$ with smooth boundary. Assume that $\varepsilon,\mu$
and $\sigma$ are in $L^\infty(\Omega)^{3\times 3}$, and they have
the following properties: There are constants $c_m,c_M>0$ such that
for all $x\in\Omega$ and arbitrary $\xi\in\mathbb{R}^3$
\begin{equation}\label{eq:regular eu}
c_m|\xi|^2\leq \xi^T\varepsilon(x)\xi\leq c_M|\xi|^2,\ \ \
c_m|\xi|^2\leq \xi^T\mu(x)\xi\leq c_M|\xi|^2
\end{equation}
and
\begin{equation}\label{eq:regular sigma}
0\leq \xi^T\sigma\xi\leq c_M|\xi|^2.
\end{equation}
We remark that the conditions (\ref{eq:regular eu})
and (\ref{eq:regular sigma}) are physical conditions for {\it
regular} EM media. We also assume that $J\in L^2(\Omega)^3$. For the
Maxwell equations (\ref{eq:Maxwell eqn}), we seek solutions
$(E,H)\in H(\mbox{curl};\Omega)\times H(\mbox{curl};\Omega)$, where
\begin{equation}
H(\mbox{curl};\Omega)=\{\mathbf{u}\in
L^2(\Omega)^3;\nabla\times\mathbf{u}\in L^2(\Omega)^3\}.
\end{equation}
We shall not give a complete review on the study of existence and
uniqueness of solutions to (\ref{eq:Maxwell eqn}) in the setting
described above, and we refer to \cite{Mon} for results related to
our present study. It is noted that there is a well-defined
continuous impedance map
\begin{equation}
\Lambda_{\varepsilon,\mu,\sigma,J}^\omega:\
H^{-1/2}(\mbox{Div};\partial\Omega)\rightarrow
H^{-1/2}(\mbox{Div};\partial\Omega),
\end{equation}
provided $\omega$ avoids a discrete set of frequencies corresponding
to {\it resonances} (cf. \cite{Mon}). Here,
\[
H^{-\frac 1 2}(\mbox{Div};\partial\Omega)=\{\mathbf{s}\in H^{-\frac
1 2}(\partial\Omega)^3;\mathbf{s}\cdot\nu=0\ \mbox{a.e. on
$\partial\Omega$}\ \mbox{and}\ \mbox{Div}\, \mathbf{s}\in H^{-\frac
1 2}(\partial\Omega)\},
\]
with $\mbox{Div}$ denoting the surface divergence on
$\partial\Omega$.

\begin{lemma}\label{lem:trans Maxwell}
Consider a transformation
$\tilde{x}=F(x):\Omega\rightarrow\tilde\Omega$, which is assumed to
be bi-Lipschitz and orientation-preserving. Let
$M=DF:=(\frac{\partial\tilde{x}_i}{\partial x_j})_{i,j=1}^3$ be the
Jacobian matrix of $F$. Assume that $(E,H)\in
H(\mbox{\emph{curl}};\Omega)\times H(\mbox{\emph{curl}};\Omega)$ are
EM fields to (\ref{eq:Maxwell eqn}), then for the pull-back fields
given by
\begin{equation}\label{eq:trans EH}
\tilde E=(F^{-1})^*E:=(M^T)^{-1}E\circ F^{-1},\ \tilde
H=(F^{-1})^*H:=(M^T)^{-1}H\circ F^{-1}
\end{equation}
and
\begin{equation}\label{eq:trans J}
\tilde J=(F^{-1})^*J=\frac{1}{\mbox{\emph{det}}(M)}M J\circ F^{-1},
\end{equation}
we have $(\tilde E, \tilde H)\in
H(\mbox{\emph{curl}};\tilde\Omega)\times
H(\mbox{\emph{curl}};\tilde\Omega)$ satisfying Maxwell's equations
\begin{equation}\label{eq:trans Maxwell}
\tilde\nabla\times\tilde E=i\omega\tilde\mu\tilde H,\ \
\tilde\nabla\times\tilde
H=-i\omega(\tilde\varepsilon+i\frac{\tilde\sigma}{\omega})\tilde
E+\tilde J\qquad \mbox{in\ \ $\tilde\Omega$},
\end{equation}
where $\tilde\nabla\times$ denotes the curl in the $\tilde
x$-coordinates, and $\tilde\varepsilon, \tilde\mu, \tilde\sigma$ are
the push-forwards of $\varepsilon,\mu, \sigma$ via $F$, defined
respectively by
\begin{align}
\tilde\varepsilon=F_*\varepsilon:=&\frac{1}{\mbox{\emph{det}}(M)}{M\cdot\varepsilon\cdot
M^T}\circ F^{-1},\label{eq:trans epsilon}\\
\tilde\mu=F_*\mu:=&\frac{1}{\mbox{\emph{det}}(M)}{M\cdot\mu\cdot
M^T}\circ F^{-1},\label{eq:trans mu}\\
\tilde\sigma=F_*\sigma:=&\frac{1}{\mbox{\emph{det}}(M)}{M\cdot\sigma\cdot
M^T}\circ F^{-1}.\label{eq:trans sigma}
\end{align}
\end{lemma}

\begin{proof}
The key ingredient for the proof of the lemma is the following
transformation rule on curl operation (see, e.g. \cite{Mon})
\begin{equation}\label{eq:trans curl}
\tilde\nabla\times \tilde E=\frac{1}{\mbox{det}(M)}M(\nabla\times
E)\circ F^{-1},\ \ \tilde\nabla\times \tilde
H=\frac{1}{\mbox{det}(M)}M(\nabla\times H)\circ F^{-1}.
\end{equation}
Using (\ref{eq:trans curl}) along with (\ref{eq:Maxwell eqn}),
(\ref{eq:trans EH}) and (\ref{eq:trans mu}), we have
\begin{equation}\label{eq:trans Maxwell1}
\begin{split}
\tilde\nabla\times\tilde E=&\frac{1}{\mbox{det}(M)}(\nabla\times
E)\circ
F^{-1}=\frac{1}{\mbox{det}(M)}M(i\omega\mu H)\circ F^{-1}\\
=& i\omega\frac{1}{\mbox{det}(M)}M\mu M^T(M^{-T}H)\circ
F^{-1}=i\omega\tilde\mu\tilde H.
\end{split}
\end{equation}
Similarly, using (\ref{eq:trans curl}), together with
(\ref{eq:Maxwell eqn}), (\ref{eq:trans EH}), (\ref{eq:trans J}),
(\ref{eq:trans epsilon}) and (\ref{eq:trans sigma}), we have
\begin{equation}
\begin{split}
\tilde\nabla\times\tilde H=&\frac{1}{\mbox{det}(M)}(\nabla\times
H)\circ F^{-1}=\frac{1}{\mbox{det}(M)}(-i\omega\varepsilon_r
E+J)\circ
F^{-1}\\
=&-i\omega\frac{1}{\mbox{det}(M)}M\varepsilon_r E\circ
F^{-1}+\frac{1}{\mbox{det}(M)}M J\circ F^{-1}\\
=&-i\omega\tilde{\varepsilon}_r\tilde E+\tilde{J},
\end{split}
\end{equation}
where
\[
\varepsilon_r=\varepsilon+i\frac{\sigma}{\omega}\quad\mbox{and}\quad
\tilde{\varepsilon}_r=\tilde{\varepsilon}+i\frac{\tilde{\sigma}}{\omega}.
\]

The proof is completed.
\end{proof}

\begin{corollary}\label{cor:trans Maxwell}
Assume that $F: \Omega\rightarrow\Omega$ is bi-Lipschitz and
orientation-preserving with $F|_{\partial\Omega}=Id$. Using Green's
identity, it is directly verified that
\begin{equation}\label{eq:cor1}
\nu\times E=\tilde\nu\times\tilde E,\ \ \nu\times
H=\tilde\nu\times\tilde H\quad\mbox{on \ $\partial\Omega$},
\end{equation}
which together with Lemma~\ref{lem:trans Maxwell} yields
\begin{equation}
\Lambda_{\varepsilon,\mu,\sigma,J}^\omega=\Lambda_{F_*\varepsilon,F_*\mu,F_*\sigma,(F^{-1})^*J}^\omega.
\end{equation}
\end{corollary}

Lemma~\ref{lem:trans Maxwell} and Corollary~\ref{cor:trans Maxwell}
summarize the basics of transformation optics in a rather general
setting, which we shall make essential use of in the present paper. In
the rest of this section, we give a short discussion on the singular
ideal cloaking device construction considered in \cite{GKLU} and
\cite{PenSchSmi} using transformation optics, and introduce the
notion of approximate cloaking from a regularization viewpoint. In
the sequel, let $B_r$ denote the ball centered at the origin with radius $r$. Let
$M_1=B_2$, $M_2=B_1$ and $M$ be the disjoint union $M=M_1\cup M_2$.
Also, let $N_1=B_2\backslash \overline{B_1}$, $N_2=B_1$ and
$N=N_1\cup N_2$. Moreover, set $\Sigma:=\partial B_1$. Consider the
map
\begin{equation}
F_1:\,M_1\backslash\{0\}\,\rightarrow\,N_1,\;\;\;\;
F_1(y)=\left(1+\frac{1}{2}|y|\right)\frac{y}{|y|},\;\;\;0<|y|<2
\end{equation}
which blows up $\{0\}$ to $N_2$ while keeps the boundary $\partial
M_1$ fixed. In \cite{GKLU} and \cite{PenSchSmi}, the authors consider
the lossless setting, i.e., one always assume that $\sigma=0$. In
the cloaking region $N_1$, the EM material parameters of the
corresponding cloaking medium are given by
\begin{equation}\label{eq:cloaking medium}\tilde\mu(x)=\tilde\eps(x)=(F_1)_*I:=\left.\displaystyle\frac{(DF_1)I(DF_1)^T}{\mbox{det}(DF_1)}\right|_{y=F_1^{-1}(x)},\;\;\;\;x\in N_1.\end{equation}
In the cloaked region $N_2=B_1$, we consider cloaking an arbitrary
but {\it regular} EM medium $(\eps_0, \mu_0)$, i.e.,
\begin{equation}
\tilde\mu(x)=\mu_0(x),\;\;\; \tilde\eps(x)=\eps_0(x)\;\;\;\;\;x\in
N_2,
\end{equation}
which can be viewed as the push-forwards of $(\mu_0, \eps_0)$ in $M_2$ by $F_2=Id$.
We denote the transformation by
\begin{equation}\label{eq:whole trans}
F=(F_1, F_2): (M_1\backslash\{0\}, M_2)\rightarrow (N_1, N_2).
\end{equation}
By (\ref{eq:cloaking medium}) together with straightforward
calculations, we have in the standard spherical coordinates
$x\mapsto (r\cos\phi\cos\theta,r\sin\phi\cos\theta,r\sin\theta)$
that
\begin{equation}\label{eq:spherical cloaking medium}
\tilde\mu=\tilde\varepsilon=2\frac{(r-1)^2}{r^2}\mathbf{e}_r+2\mathbf{e}_\theta,\quad
1<r<2,
\end{equation}
where $\mathbf{e}_r$ and $\mathbf{e}_\theta$ are respectively, the
unit projections along radial and angular directions, i.e.,
\[\mathbf{e}_r=I-\hat{x}\hat{x}^T,\quad \mathbf{e}_\theta=\hat{x}\hat{x}^T,\quad \hat{x}=\frac{x}{|x|}.\]
 It is readily seen
that as one approaches the cloaking interface $\Sigma$ the cloaking
medium becomes singular, since $\tilde\varepsilon$ and $\tilde\mu$
no longer satisfy the condition (\ref{eq:regular eu}).
Finite energy solutions to the singular Maxwell's equations
underlying the cloaking are investigated in \cite{GKLU}. It is shown
that $\{0\}$ is a removable singular point. Specifically, let
$(\tilde E,\tilde H)$ be the EM fields corresponding to
$\{N;\tilde\varepsilon,\tilde\mu\}$, then
$(E^+,H^+)=(F_1)^*(\tilde E, \tilde H)$ are EM fields in free
space on $M_1$, which implies by Corollary~\ref{cor:trans Maxwell}
that
\[
\Lambda_{\tilde\varepsilon,\tilde\mu}^\omega=\Lambda_0^\omega.
\]
On the other hand, $(E^-,H^-)=(F_2)^*(\tilde E,\tilde H)$
satisfy the Maxwell equations
\begin{equation}\label{eq:interior problem}
\begin{cases}
& \nabla\times E^{-}=i\omega\mu_0 H^{-},\quad\nabla\times H^{-}=-i\omega\varepsilon_0 E^- \ \ \mbox{on $M_2$}\\
&\nu\times E^{-}=0,\quad \nu\times H^-=0 \ \ \mbox{on $\partial
M_2$}.
\end{cases}
\end{equation}
Generically, one would have $E^-=H^-=0$ for (\ref{eq:interior
problem}) due to the homogeneous `hidden' PEC and PMC boundary
conditions in \eqref{eq:interior problem} on $\partial M_2$. Also, due to such `hidden' boundary
conditions, it is claimed in \cite{GKLU} that one cannot perfectly
cloak a generic internal current in the cloaked region $B_1$.

As can be seen from (\ref{eq:spherical cloaking medium}) the
cloaking medium for the ideal cloaking is singular, which poses challenges to both mathematical analysis and physical realization. In
order to construct practical nonsingular cloaking devices, it is
natural to incorporate regularization by considering approximate
cloaking, which we shall investigate in the subsequent sections. We
conclude this section by introducing the notion of approximate EM
cloaking.

\begin{definition}\label{def:2}
Let $\Omega$ and $D$ be bounded domains in $\mathbb{R}^3$ with
$D\Subset\Omega$, representing respectively the cloaking region and
the cloaked region. Let $\rho>0$ denote a parameter and $e(\rho)$
be a positive function such that
\[
e(\rho)\rightarrow 0\quad\mbox{as\ $\rho\rightarrow 0^+$}.
\]
$\{\Omega\backslash\bar{D}; \varepsilon_c^\rho,\mu_c^\rho,
\sigma_c^\rho\}$ is said to be an {\it approximate invisibility
cloaking} for the region $D$ if
\begin{equation}\label{eq:near invisibility}
\|\Lambda_{\varepsilon_e^\rho,\mu_e^\rho,\sigma_e^\rho,J_e}^\omega-\Lambda_{0}^\omega\|
=e(\rho)\quad\mbox{as $\rho\rightarrow 0^+$},
\end{equation}
where the extended object $\{\Omega;
\varepsilon_e^\rho,\mu_e^\rho,\sigma_e^\rho, J_e\}$ is defined
similarly to (\ref{eq:extended object}) by replacing
$\varepsilon_c,\mu_c,\sigma_c$ with $\varepsilon_c^\rho,\mu_c^\rho,
\sigma_c^\rho$.
\end{definition}

According to (\ref{eq:near invisibility}), with the cloaking device
$\{\Omega\backslash\bar{D}; \varepsilon_c^\rho,\mu_c^\rho,
\sigma_c^\rho\}$ we shall have the `near-invisibility' cloaking
effect. In order for the invisibility cloaking and approximate
invisibility cloaking in Definitions~\ref{def:1} and \ref{def:2}
make the right sense, throughout the rest of the paper, we always
assume that there is a well-defined impedance map $\Lambda_0^\omega$
in the free space; namely, it is assumed that there is no resonance
occurring in the free space.

\section{Nonsingular approximate cloaking of passive medium}

In this section, we consider the approximate EM cloaking for a
relatively simpler case by assuming that all the EM media concerned
are lossless, i.e. $\sigma=0$, and also there is no source/sink
present, i.e. $J=0$.

For approximate acoustic cloaking by regularization, Kohn et al., in
\cite{KOVW}, proposed blowing up a small ball $B_\rho$ to $B_1$
using a nonsingular transformation $F_\rho$ which degenerates to the
singular transformation $F$ in (\ref{eq:whole trans}) as
$\rho\rightarrow0$, while Greenleaf et al., in \cite{GKLU_2},
proposed blowing up $B_\rho$ to $B_R$ with $R>1$ by the original
singular transformation $F$. For the present study on approximate EM
cloaking, we shall focus on the `blow-up-a-small-ball-to-$B_1$' scheme and
evaluate its performance. However, it is remarked
that the other regularization scheme has been verified to yield the
same performances for approximate EM cloaking.

\subsection{Construction of approximate EM
cloaking}\label{subsect:31}

Let $0<\rho<1$ denote a regularizer and
\begin{equation}\label{eq:coefficients}
a=\frac{2(1-\rho)}{2-\rho},\;\;\;\;b=\frac{1}{2-\rho}.
\end{equation}
Consider the nonsingular transformation from $B_2$ to $B_2$ defined
by
\begin{equation}\label{eq:transf}
x:=F_\rho(y)=\left\{\begin{array}{ll}F_\rho^{(1)}(y)=(a+b|y|)\frac{y}{|y|}
&\rho<|y|<2,\\
F_\rho^{(2)}(y)=\frac{y}{\rho}&|y|\leq\rho.\end{array}\right.
\end{equation}
Our approximate cloaking device is obtained by the push-forward of
a homogeneous medium in $B_2\backslash \overline{B_\rho}$ by
$F_\rho^{(1)}$. Suppose we hide a regular but arbitrary uniform EM
medium $(\eps_0,\mu_0)$ in the cloaked region $B_1$. Then the
corresponding EM material parameter in $B_2$ is
\[(\tl\eps_\rho(x),\tl\mu_\rho(x))=\left\{\begin{array}{ll}((F_\rho^{(1)})_*I, (F_\rho^{(1)})_*I)\;\;&1<|x|<2,\\ (\eps_0,\mu_0)&|x|<1\end{array}\right.\]
which are obviously nonsingular. The EM fields $(\tl E_\rho, \tl
H_\rho)\in H(\mbox{curl};B_2)\times H(\mbox{curl};B_2)$
corresponding to $\{B_2;\tl\varepsilon_\rho,\tilde\mu_\rho\}$
satisfy Maxwell's equations
\begin{equation}\label{Max_apprx_1}
\left\{\begin{array}{l}
\nabla\times\tl E_\rho=i\omega\tl \mu_\rho(x)\tl H_\rho,\;\;\;\;\nabla\times\tl H_\rho=-i\omega\tl\eps_\rho(x)\tl E_\rho\;\;\;\;\mbox{in }\;B_2,\\
\nu\times\tl E_\rho|_{\partial B_2}=f\in
H^{-1/2}(\mbox{Div};\partial B_2).
\end{array}\right.
\end{equation}
By Lemma~\ref{lem:trans Maxwell}, the pull-back EM fields
\[(E_\rho, H_\rho)=
((F_\rho)^*\tl E_\rho, (F_\rho)^*\tl H_\rho)\in
H(\mbox{curl};B_2)\times H(\mbox{curl};B_2)
\]
satisfy Maxwell's equations
\begin{equation}\label{Max_apprx_2}
\left\{\begin{array}{l}
\nabla\times E_\rho=i\omega\mu_\rho(y) H_\rho,\;\;\;\;\nabla\times H_\rho=-i\omega\eps_\rho(y)E_\rho,\;\;\;\;\mbox{in }B_2\backslash \overline{B_\rho}, \\
\nu\times E_\rho|_{\partial B_2}=f\in H^{-1/2}(\mbox{Div};\partial
B_2),
\end{array}\right.
\end{equation}
where
\[(\eps_\rho(y),\mu_\rho(y))=\left\{\begin{array}{ll}(I, I)\;\;&\rho<|y|<2,\\((F_\rho^{(2)})^*\eps_0,(F_\rho^{(2)})^*\mu_0)\;\;&|y|<\rho.\end{array}\right.\]
By Corollary~\ref{cor:trans Maxwell}, we see that
\[\Lambda_{\eps_\rho,\mu_\rho}^\omega=\Lambda_{\tl \eps_\rho,\tl \mu_\rho}^\omega.\]
Hence, the estimate of
$\Lambda_{\tilde\varepsilon_\rho,\tilde\mu_\rho}^\omega$ for the
approximate EM cloaking is the same to that of
$\Lambda_{\varepsilon_\rho,\mu_\rho}^\omega$.
\subsection{Convergence and hidden boundary conditions}

Henceforth, the following notations for EM fields shall be adopted
\[\tl{E}_\rho:=(\tl{E}_\rho^{+}, \tl{E}_\rho^{-}),\;\;\;\; \tl{H}_\rho:=(\tl{H}_\rho^{+}, \tl{H}_\rho^{-})\;\;\;\mbox{for } \;x\in (B_2\backslash \overline{B_1}, B_1)\] and
\[E_\rho:=(E_\rho^{+}, E_\rho^{-}),\;\;\;\; H_\rho:=
(H_\rho^{+}, H_\rho^{-})\;\;\;\mbox{for }\; y\in
(B_2\backslash\overline{B_\rho}, B_\rho).\] We also use $\tilde E,
\tilde H$ to represent the finite-energy EM fields considered in
\cite{GKLU} for singular ideal cloaking which we discussed earlier
in Section~2. \eqref{Max_apprx_1} and \eqref{Max_apprx_2} can be
reformulated as the following transmission problems
\begin{equation}\label{Max_apprx_3}
\left\{\begin{array}{l}\nabla\times \tl E^+_\rho=i\omega\tl\mu_\rho(x)\tl H^+_\rho,\;\;\;\;\nabla\times\tl H^+_\rho=-i\omega\tl \eps_\rho(x)\tl E^+_\rho\;\;\;\;\mbox{in }\; B_2\backslash\overline{B_1},\\
\nabla\times\tl E^-_\rho=i\omega\mu_0\tl H^-_\rho,\;\;\;\;\nabla\times\tl H^-_\rho=-i\omega\eps_0\tl E^-_\rho\;\;\;\;\mbox{in }\;B_1,\\
\nu\times\tl E^+_\rho|_{\Sigma^+}=\nu\times\tl E^-_\rho|_{\Sigma^-},\;\;\;\;\nu\times\tl H^+_\rho|_{\Sigma^+}=\nu\times\tl H^-_\rho|_{\Sigma^-},\\
\nu\times\tl E^+_\rho|_{\partial B_2}=f.\end{array}\right.
\end{equation}
and
\begin{equation}\label{Max_apprx_4}
\left\{\begin{array}{l}\nabla\times E^+_\rho=i\omega H^+_\rho,\;\;\;\;\nabla\times H^+_\rho=-i\omega E^+_\rho\;\;\;\;\mbox{in }\; B_2\backslash\overline{B_\rho},\\
\nabla\times E^-_\rho=i\omega\mu_\rho(y)H^-_\rho,\;\;\;\;\nabla\times H^-_\rho=-i\omega\eps_\rho(y)E^-_\rho\;\;\;\;\mbox{in }\;B_\rho,\\
\nu\times E^+_\rho|_{\Sigma_\rho^+}=\nu\times E^-_\rho|_{\Sigma_\rho^-},\;\;\;\;\nu\times H^+_\rho|_{\Sigma_\rho^+}=\nu\times H^-_\rho|_{\Sigma_\rho^-},\\
\nu\times E^+_\rho|_{\partial B_2}=f.\end{array}\right.
\end{equation}
where $\Sigma_\rho:=\partial B_\rho$.

Our arguments rely heavily on expanding the EM fields into series of
spherical wave functions. To that end, we introduce for
$n\in\mathbb{Z}^+$ and $m\in\mathbb{Z}$,
\[
M_{n, \zeta}^m(x):=\nabla\times\{xj_n(\zeta|x|)Y_n^m(\hat{x})\},\quad N_{n,\zeta}^m(x):=\nabla\times\{xh_n^{(1)}(\zeta|x|)Y_n^m(\hat{x})\},
\]
where $\zeta\in\mathbb{C}$ is a complex number and $\hat{x}=x/|x|$
for $x\in\mathbb{R}^3$. Here, $Y_n^m(\hat{x})$ are spherical
harmonics and, $h_n^{(1)}(z):=j_n(z)+iy_n(z)$ with $j_n(z)$ and
$y_n(z)$, for $z\in\C$, being the spherical Bessel functions of the first and
second kind, respectively.
The most important property of such functions for our argument are their asymptotical behavior with respect to small variables:
\begin{equation}
j_n(z)=\mO(|z|^n),\quad h_n(z)=\mO(|z|^{-n-1}),\quad\mbox{for }|z|\ll1.
\end{equation}
We refer to \cite{CK} and \cite{Mon} for more properties of the functions introduced here.
%In \cite{CK}, sets $\{M_{n,k\omega}^m, \nabla\times M_{n,k\omega}^m\}_{n,m}$, $\{N_{n,k\omega}^m, \nabla\times N_{n,k\omega}^m\}_{n,m}$
%generate the entire solution and radiating solution of homogeneous Maxwell's equations.

The second Maxwell's equations in \eqref{Max_apprx_3} and the first
of \eqref{Max_apprx_4} would give rise to waves for $x\in B_1$
\begin{equation}\label{SHE_-}
\left\{\begin{array}{l}\displaystyle\tl
E^-_\rho=\eps_0^{-1/2}\sum_{n=1}^{\infty}\sum_{m=-n}^{n}\alpha_n^mM_{n,k\omega}^m+\beta_n^m\nabla\times
M_{n,k\omega}^m,\\ \displaystyle\tl
H^-_\rho=\frac{1}{ik\omega}\mu_0^{-1/2}\sum_{n=1}^{\infty}\sum_{m=-n}^{n}k^2\omega^2\beta_n^mM_{n,k\omega}^m+\alpha_n^m\nabla\times
M_{n,k\omega}^m,\end{array}\right.\end{equation} and for $y\in
B_2\backslash\overline{B_\rho}$
\begin{equation}\label{plbk}
\left\{\begin{array}{l}\displaystyle E^+_\rho=\sum_{n=1}^\infty\sum_{m=-n}^nc_n^mN_{n,\omega}^m+d_n^m\nabla\times N_{n,\omega}^m+\gamma_n^mM_{n,\omega}^m+\eta_n^m\nabla\times M_{n,\omega}^m,\\
\displaystyle
H^+_\rho=\frac{1}{i\omega}\sum_{n=1}^\infty\sum_{m=-n}^n\omega^2d_n^mN_{n,\omega}^m+c_n^m\nabla\times
N_{n,\omega}^m+\omega^2\eta_n^mM_{n,\omega}^m+\gamma_n^m\nabla\times
M_{n,\omega}^m,\end{array}\right.
\end{equation}
where $k=(\mu_0\eps_0)^{1/2}$.

%\begin{remark}Notice here $((F_\rho^{(2)})^*\mu_0, (F_\rho^{(2)})^*\eps_0))=\frac{1}{\rho}(\mu_0, \eps_0)$ is constant in $B_\rho$.
%Therefore, we can also write down the series expansion for $(E_\rho^-, H_\rho^-)$ in $B_\rho$ without writing the one for $(\tl E_\rho^-, \tl H_\rho^-)$
%in $B_1$. And the following argument works. The reason we use the latter expansion is the following argument also works in the case
%we don't know what the transformation in $B_\rho$ is.
%\end{remark}

The following lemma characterizes the asymptotic behaviors of
the coefficients in the spherical expansions \eqref{SHE_-} and
\eqref{plbk} as $\rho$ goes to zero.
\begin{lemma}\label{lem: cal_lossless}
Assume $\omega$ is not an eigenvalue of \eqref{Max_apprx_3}, namely,
the corresponding homogeneous equations have only trivial solutions.
Let $(\tl E_\rho, \tl H_\rho)$ be the unique solutions to
\eqref{Max_apprx_3}, whereas $(E_\rho, H_\rho)=((F_\rho)^*\tl
E_\rho, (F_\rho)^*\tl H_\rho)$ be the unique solutions to
\eqref{Max_apprx_4}. $(\tl E_\rho^-, \tl H_\rho^-)$ and $(E_\rho^+,
H_\rho^+)$ are given by \eqref{SHE_-} and \eqref{plbk},
respectively, whose coefficients are uniquely determined by the
boundary data $f$. As $\rho\rightarrow 0^+$, we have
\begin{equation}\label{ga_et_c_d}
\gamma_n^m=\mO(1),\;\;\; \eta_n^m=\mO(1);\;\;\;
c_n^m=\mO(\rho^{2n+1}),\;\;\; d_n^m=\mO(\rho^{2n+1}),
\end{equation}
and
\begin{equation}\label{al_be}
\alpha_n^m=\mO(\rho^{n+1}),\;\;\;
\beta_n^m=\mO(\rho^{n+1}).\end{equation}
\end{lemma}
\begin{proof}
We need to introduce the vector spherical harmonics
\[
U_n^m:={\frac{1}{\sqrt{n(n+1)}}}\mbox{Grad }Y_n^m,\;\;\;\;\;
V_n^m:=\nu\times U_n^m,
\]
where Grad denotes the surface gradient. Define
\[\mathcal{H}_n(t):=h^{(1)}_n(t)+t{h^{(1)}_n}'(t),\;\;\;\;\;\mathcal{J}_n(t):=j_n(t)+tj'_n(t).\] For $t\ll1$,
one can verify $\mathcal{J}_n(t)=\mO(t^n)$ and $\mathcal{H}_n(t)=\mO(t^{-n-1})$.
Then on a sphere $\partial B_R$, we have for $0<R<1$,
{\begin{equation}\label{bdry_1} \left\{\begin{array}{l}\displaystyle
\nu\times \tl E_\rho^-|_{\partial
B_R}=\eps_0^{-1/2}\sum_{n=1}^{\infty}\sum_{m=-n}^{n}\sqrt{n(n+1)}\Big(\alpha_n^mj_n(k\omega
R)U_n^m\\
\hspace{6cm}+\beta_n^m\frac{1}{R}\mathcal{J}_n(k\omega R)V_n^m\Big),\\
\displaystyle \nu\times\tl{H}_\rho^{-}|_{\partial
B_R}=\frac{1}{ik\omega}\mu_0^{-1/2}\sum_{n=1}^{\infty}\sum_{m=-n}^{n}{\sqrt{n(n+1)}}\Big(\beta_n^mk^2\omega^2j_n(k\omega
R)U_n^m\\
\hspace{6cm}+\alpha_n^m\frac{1}{R}\mathcal{J}_n(k\omega
R)V_n^m\Big),\end{array}\right.
\end{equation}}
whereas for $\rho<R<2$, {\begin{equation}\label{bdry_2}
\left\{\begin{array}{l}\displaystyle \nu\times E_\rho^+|_{\partial B_R}=\sum_{n=1}^{\infty}\sum_{m=-n}^{n}{\sqrt{n(n+1)}}\Big(c_n^m h_n^{(1)}(\omega R)U_n^m+d_n^m\frac{1}{R}\mathcal{H}_n(\omega R)V_n^m\\
\hspace{2cm}+\gamma_n^mj_n(\omega R)U_n^m+\eta_n^m\frac{1}{R}\mathcal{J}_n(\omega R)V_n^m\Big),\\
\displaystyle \nu\times H_\rho^+|_{\partial
B_R}=\frac{1}{i\omega}\sum_{n=1}^{\infty}\sum_{m=-n}^{n}{\sqrt{n(n+1)}}\Big(\omega^2d_n^mh^{(1)}_n(\omega R)U_n^m\\
\hspace{2cm}+c_n^m\frac{1}{R}\mathcal{H}_n(\omega R)V_n^m+\omega^2\eta_n^mj_n(\omega
R)U_n^m+\gamma_n^m\frac{1}{R}\mathcal{J}_n(\omega R)V_n^m\Big).
\end{array}\right.
\end{equation}}
Expanding the boundary value on $\partial B_2$ in terms of the
vector spherical harmonics, we have
\begin{equation}\label{f_exp}
f=
\sum_{n=1}^{\infty}\sum_{m=-n}^{n}{\sqrt{n(n+1)}}(f_{nm}^{(1)}U_n^m+f_{nm}^{(2)}V_n^m),
\end{equation} the boundary condition
$\nu\times E^{+}_\rho|_{\partial B_2}=f$ implies
\[\mbox{(R-1)}\;\;\left\{\begin{array}{l}c_n^mh_n^{(1)}(2\omega)+\gamma_n^mj_n(2\omega)=f_{nm}^{(1)},\\
d_n^m\mathcal{H}_n(2\omega)+\eta_n^m\mathcal{J}_n(2\omega)=2f_{nm}^{(2)}.\end{array}\right.\]
Since $\tl E_\rho=(F_\rho^{-1})^*E_\rho$, the transmission condition
on the electric field in \eqref{Max_apprx_3} reads
\[\nu\times\tl E^+_\rho|_{\Sigma^+}=\rho(\nu\times E^+_\rho|_{\Sigma^+_\rho})=\nu\times\tl E^-_\rho|_{\Sigma^-}.
\]
By \eqref{bdry_1} and \eqref{bdry_2}, we have
\[\mbox{(R-2)}\;\;\left\{\begin{array}{l}\rho c_n^mh_n^{(1)}(\omega\rho)+\rho\gamma_n^mj_n(\omega\rho)=\eps_0^{-1/2}\alpha_n^m j_n(k\omega),\\
d_n^m\mathcal{H}_n(\omega\rho)+\eta_n^m\mathcal{J}_n(\omega\rho)=\eps_0^{-1/2}\beta_n^m\mathcal{J}_n(k\omega).
\end{array}\right.\]
Similarly, the transmission condition on the magnetic field implies
\[\mbox{(R-3)}\;\;\left\{\begin{array}{l}k c_n^m\mathcal{H}_n(\omega\rho)+k \gamma_n^m\mathcal{J}_n(\omega\rho)=\mu_0^{-1/2}\alpha_n^m\mathcal{J}_n(k\omega),\\
\rho
d_n^mh_n^{(1)}(\omega\rho)+\rho\eta_n^mj_n(\omega\rho)=\mu_0^{-1/2}k\beta_n^m
j_n(k\omega).
\end{array}\right.\]
By (R-2) and (R-3), we have
\begin{equation}\label{al_be_c_d_nSrc}
c_n^m=t_1\gamma_n^m,\;\;\;\alpha_n^m=t_2\gamma_n^m,\;\;\;d_n^m=t_3\eta_n^m,\;\;\;\beta_n^m=t_4\eta_n^m,\end{equation}
where as $\rho\rightarrow0^+$,
\begin{equation}\label{t}\begin{array}{l}
t_1:=\displaystyle\frac{\eps_0^{-1/2}k\mathcal{J}_n(\omega\rho)j_n(k\omega)-\mu_0^{-1/2}\rho j_n(\omega\rho)\mathcal{J}_n(k\omega)}{\mu_0^{-1/2}\rho h_n^{(1)}(\omega\rho)\mathcal{J}_n(k\omega)-\eps_0^{-1/2}k\mathcal{H}_n(\omega\rho)j_n(k\omega)}=\mO(\rho^{2n+1}),\\
t_2:=\displaystyle\frac{k\rho\mathcal{J}_n(\omega\rho)h_n^{(1)}(\omega\rho)-k\rho j_n(\omega\rho)\mathcal{H}_n(\omega\rho)}{\mu_0^{-1/2}\rho\mathcal{J}_n(k\omega)h_n^{(1)}(\omega\rho)-\eps_0^{-1/2}k j_n(k\omega)\mathcal{H}_n(\omega\rho)}=\mO(\rho^{n+1}),\\
t_3:=\displaystyle\frac{\mu_0^{-1/2}k \mathcal{J}_n(\omega\rho)j_n(k\omega)-\eps_0^{-1/2}\rho j_n(\omega\rho)\mathcal{J}_n(k\omega)}{\eps_0^{-1/2}\rho h_n^{(1)}(\omega\rho)\mathcal{J}_n(k\omega)-\mu_0^{-1/2}k\mathcal{H}_n(\omega\rho)j_n(k\omega)}=\mO(\rho^{2n+1}),\\
t_4:=\displaystyle\frac{\rho\mathcal{J}_n(\omega\rho)h_n^{(1)}(\omega\rho)-\rho
j_n(\omega\rho)\mathcal{H}_n(\omega\rho)}{\eps_0^{-1/2}\rho\mathcal{J}_n(k\omega)h_n^{(1)}(\omega\rho)-\mu_0^{-1/2}k
j_n(k\omega)\mathcal{H}_n(\omega\rho)}=\mO(\rho^{n+1})
\end{array}\end{equation}
By (R-1), we have
\begin{equation}\label{ga_et}
\gamma_n^m=\displaystyle\frac{f_{nm}^{(1)}}{t_1h_n^{(1)}(2\omega)+j_n(2\omega)}=\mO(1),\;\;\;\;
\eta_n^m=\displaystyle\frac{2f_{nm}^{(2)}}{t_3\mathcal{H}_n(2\omega)+\mathcal{J}_n(2\omega)}=\mO(1).\end{equation}
By \eqref{al_be_c_d_nSrc}, these further imply
\[\alpha_n^m=\mO(\rho^{n+1}),\;\;\; \beta_n^m=\mO(\rho^{n+1}),\;\;\;c_n^m=\mO(\rho^{2n+1}),\;\;\; d_n^m=\mO(\rho^{2n+1}).\]
\end{proof}

%{\color{red}\begin{remark}
%From (R-2-1) and (R-3-1), we have the relations of convergence orders in $\rho$ as following
%\[\begin{split}\mO(\rho^{-n})c_n^m+\mO(\rho^{n+1})\gamma_n^m=\mO(1)\alpha_n^m,\\
%\mO(\rho^{-n-1})c_n^m+\mO(\rho^n)\gamma_n^m=\mO(1)\alpha_n^m.\end{split}\]
%With the coefficients above, both sides of the first equation balance. For the second equation, the LHS is of $\mO(\rho^n)$
%while the RHS is of $\mO(\rho^{n+1})$, which appears not balanced. In fact, one can show the LHS is of order $\mO(\rho^{n+1})$:
%\[\begin{split}\mbox{LHS}=&\left(\omega t_1\mathcal{H}_n(k\rho)+\omega\mathcal{J}_n(k\rho)\right)\gamma_n^m\\
%=&\displaystyle\frac{\mu_0^{-1/2}\rho\mathcal{J}_n(k\omega)[j_n(k\rho)\mathcal{H}_n(k\rho)-h_n^{(1)}(k\rho)
%\mathcal{J}_n(k\rho)]}{\mu_0^{-1/2}\rho h_n^{(1)}(k\rho)\mathcal{J}_n(k\omega)-\eps_0^{-1/2}\omega\mathcal{H}_n(k\rho)j_n(k\omega)}\\
%=&\displaystyle\frac{\mu_0^{-1/2}k\rho^2\mathcal{J}_n(k\omega)[j_n(k\rho){h^{(1)}_n}'(k\rho)-h_n^{(1)}(k\rho)j'_n(k\rho)]}
%{\mu_0^{-1/2}\rho h_n^{(1)}(k\rho)\mathcal{J}_n(k\omega)-\eps_0^{-1/2}\omega\mathcal{H}_n(k\rho)j_n(k\omega)}=\mO(\rho^{n+1}).
%\end{split}
%\]
%\end{remark}}

We are in a position to evaluate the approximate EM cloaking. Our
observations are summarized in the following.

\begin{proposition}\label{prop:approximate cloaking}
For the approximate EM cloaking, if $\omega$ is not an eigenvalue of
\eqref{Max_apprx_3}, we have
\begin{equation}\label{eq:boundary norm estimate}
\|\Lambda_{\tilde\varepsilon_\rho,\tilde\mu_\rho}^\omega-
\Lambda_0^\omega\|=\mathcal{O}(\rho^3)\quad\mbox{as\
$\rho\rightarrow 0^+$},
\end{equation}
where $\|\cdot\|$ denotes the operator norm of the impedance map.
\end{proposition}
\begin{proof}
We write the EM fields $(E, H)$ propagating in the free space as
\begin{equation}\label{vacuum}
\left\{\begin{array}{l}\displaystyle E=\sum_{n=1}^\infty\sum_{m=-n}^n a_n^mM_{n, \omega}^m+b_n^m\nabla\times M_{n, \omega}^m,\\
H=\displaystyle\frac{1}{i\omega}\sum_{n=1}^\infty\sum_{m=-n}^n\omega^2b_n^mM_{n,
\omega}^m+a_n^m\nabla\times M_{n, \omega}^m.
\end{array}\right.
\end{equation}
Consider the boundary condition $\nu\times E|_{\partial B_2}=f$
satisfied by the $(E,H)$ fields with $f$ given by \eqref{f_exp}. By
straightforward calculations, we have
\[a_n^m=\frac{f_{nm}^{(1)}}{j_n(2\omega)},\;\;\;\;\; b_n^m=\frac{2f_{nm}^{(2)}}{\mathcal{J}_n(2\omega)}.\]
Hence, the tangential magnetic field on the boundary is given by
\begin{equation}\label{eq:ideal H field}
\nu\times H|_{\partial
B_2}=\frac{1}{i\omega}\sum_{n=1}^\infty\sum_{m=-n}^n{\sqrt{n(n+1)}}\left(b_n^m\omega^2j_n(2\omega)U_n^m
+\frac{1}{i\omega}a_n^m\frac{1}{2}\mathcal{J}_n(2\omega)V_n^m\right).
\end{equation}
Compared to $\nu\times H^+_\rho|_{\partial B_2}$ from
\eqref{bdry_2}, one observes that $c_n^m$, $d_n^m$,
$\gamma_n^m-a_n^m$ and $\eta_n^m-b_n^m$ approach zero of order
$\mO(\rho^{2n+1})$, which in turn implies (\ref{eq:boundary norm
estimate}).
\end{proof}

Proposition~\ref{prop:approximate cloaking} shows that the
approximate cloaking scheme constructed in Section~\ref{subsect:31}
actually gives the near-invisibility cloaking effect. Next, we
consider the limiting state of the approximate cloaking, showing
that it converges to the singular ideal cloaking.

\begin{proposition}\label{prop:convergence}
For the approximate EM cloaking, if $\omega$ is not an eigenvalue of
\eqref{Max_apprx_3}, we have
\begin{equation}\label{eq:convergence}
\tilde{E}_\rho\rightarrow \tilde E\quad\mbox{and}\quad
\tilde{H}_\rho\rightarrow \tilde H\quad\mbox{as\ $\rho\rightarrow
0^+$}.
\end{equation}
\end{proposition}

\begin{proof}
{We first show
\begin{equation}\label{eq:convergence1}
(E_\rho^+,H_\rho^+)=(F_\rho^{(1)})^*(\tilde{E}_\rho^+,\tilde{H}_\rho^+)\rightarrow
(E, H)\quad\mbox{as\ $\rho\rightarrow 0^+$}.
\end{equation}
It is easily verified that on any compact subset of $B_2$ away from
the origin, one has that $(E_\rho^+, H_\rho^+)$ converges to $(E,
H)$ at the rate $\mathcal{O}(\rho^3)$. Indeed, we shall show
\[
\|E_\rho^+-E\|_{L^2(B_2\backslash\overline{B_\rho})}+
\|H_\rho^+-H\|_{L^2(B_2\backslash\overline{B_\rho})}=\mathcal{O}(\rho^{3/2}),
\]
which implies (\ref{eq:convergence1}).
% or equivalently
%\[\int_{B_2\backslash\overline{B_\rho}}|E^+_\rho-E|^2+|H^+_\rho-H|^2dx\rightarrow0\quad\mbox{as}\,\rho\rightarrow 0^+.\]
To that end, we note the following identities
\begin{equation}\label{eq:spherical_harmonics}\left\{\begin{array}{l}M_{n,\omega}^m(x)=-\sqrt{n(n+1)}j_n(\omega|x|)V_n^m(\hat x),\\
 N_{n,\omega}^m(x)=-\sqrt{n(n+1)}h_n^m(\omega|x|)V_n^m(\hat x),\\
\nabla\times M_{n,\omega}^m(x)=\displaystyle\frac{\sqrt{n(n+1)}}{|x|}\mathcal{J}_n(\omega|x|)U_n^m(\hat x)+\frac{n(n+1)}{|x|}j_n(\omega|x|)Y_n^m(\hat x)\hat x,\\
\nabla\times N_{n,\omega}^m(x)=\displaystyle\frac{\sqrt{n(n+1)}}{|x|}\mathcal{H}_n(\omega|x|)U_n^m(\hat x)+\frac{n(n+1)}{|x|}h_n^{(1)}(\omega|x|)Y_n^m(\hat x)\hat x.
\end{array}\right.\end{equation}
By \eqref{plbk} and \eqref{vacuum}, we have
\[\begin{split}
E_\rho^+-E=\sum_{n=1}^\infty\sum_{m=-n}^n&-\sqrt{n(n+1)}[(\gamma_n^m-a_n^m)j_n(\omega|x|)+c_n^mh_n^{(1)}(\omega|x|)]V_n^m(\hat x)\\
&+\frac{\sqrt{n(n+1)}}{|x|}[(\eta_n^m-b_n^m)\mathcal{J}_n(\omega|x|)+d_n^m\mathcal{H}_n(\omega|x|)]U_n^m(\hat x)\\
&+\frac{n(n+1)}{|x|}[(\eta_n^m-b_n^m)j_n(\omega|x|)+d_n^mh_n^{(1)}(\omega|x|)]Y_n^m(\hat x)\hat x.
\end{split}\]
This implies as $\rho\rightarrow0^+$
\[\begin{split}
&\int_{B_2\backslash\overline{B_\rho}}|E^+_\rho-E|^2dx\\
&\hspace{1.5cm}=\sum_{n=1}^\infty\sum_{m=-n}^n\int_\rho^2n(n+1)|(\gamma_n^m-a_n^m)j_n(\omega r)+c_n^mh_n^{(1)}(\omega r)|^2r^2dr\\
&\hspace{3.2cm}+\int_\rho^2n(n+1)|(\eta_n^m-b_n^m)\mathcal{J}_n(\omega r)+d_n^m\mathcal{H}_n(\omega r)|^2dr\\
&\hspace{3.2cm}+\int_\rho^2n^2(n+1)^2|(\eta_n^m-b_n^m)j_n(\omega r)+d_n^mh_n^{(1)}(\omega r)|^2dr\\
&\hspace{1.5cm}=\sum_{n=1}^\infty\sum_{m=-n}^n\mO(\rho^{2n+1})=\mO(\rho^3)
\end{split}\]
by the convergence orders of the coefficients. Similarly, we have
\[\int_{B_2\backslash\overline{B_\rho}}|H^+_\rho-H|^2dx=\mO(\rho^3).\]}

On the other hand, it is observed from \eqref{bdry_1} that
{\begin{equation}\left\{\begin{array}{l}\nu\times \tl E_\rho^-|_{\Sigma^-}=\displaystyle\eps_0^{-1/2}\sum_{n=1}^{\infty}\sum_{m=-n}^{n}\sqrt{n(n+1)}\Big(\alpha_n^mj_n(k\omega)U_n^m\\
\hspace{6cm}+\beta_n^m\mathcal{J}_n(k\omega)V_n^m\Big),\\
\nu\times\tl{H}_\rho^{-}|_{\Sigma^-}=\displaystyle\frac{1}{ik\omega}\mu_0^{-1/2}\displaystyle\sum_{n=1}^{\infty}\sum_{m=-n}^{n}\sqrt{n(n+1)}
\Big(\beta_n^mk^2\omega^2j_n(k\omega)U_n^m\\
\hspace{6cm}+\alpha_n^m\mathcal{J}_n(k\omega)V_n^m\Big)
\end{array}\right.\end{equation}}
are both of $\mO(\rho^2)$ as $\rho\rightarrow 0^+$.
Therefore the homogeneous PEC and PMC conditions naturally appears
on the interior cloaking interface $\Sigma^-$. This is consistent
with the interior `hidden' boundary conditions discovered in
\cite{GKLU} for singular ideal cloaking (see also (\ref{eq:interior
problem})), which together with (\ref{eq:convergence1}) implies
(\ref{eq:convergence}).
\end{proof}

\subsection{Cloak-busting inclusions and frequency dependence}
{\color{black} In our earlier discussion, we achieved
near-invisibility under the condition that there are no resonances
occurring. That is, $\omega$ is not an eigenvalue to
(\ref{Max_apprx_1}), or equivalently, to (\ref{Max_apprx_2}). In
fact, if $\omega$ is an eigenvalue to (\ref{Max_apprx_2}), the small
inclusion $(\varepsilon_\rho,\mu_\rho)$ in the free space could have
a large effect on the boundary measurement. In this resonance case,
one may not even has a well-defined boundary operator
$\Lambda_{\varepsilon_\rho,\mu_\rho}^\omega$. The failure of the
near-invisible cloaking due to such ``cloak-busting" inclusion is
also observed in the study of approximate acoustic cloaking in
\cite{KOVW}. In the following, we shall show a similar result that
for a fixed approximate EM cloaking scheme, there always exists
certain interior content $(\eps_0, \mu_0)$ such that resonance
occurs at certain frequency $\omega$. We shall be looking for such
triples $(\omega, \eps_0, \mu_0)$, or equivalently $(\omega, k,
\mu_0)$, dependent on $\rho$, such that \eqref{Max_apprx_3} is
ill-posed. It corresponds to choices of $(\omega, k, \mu_0)$ such
that the coefficient matrices of systems (R-1), (R-2) and (R-3) are
singular.

We consider two decoupled systems (R-1-1)--(R-2-1)--(R-3-1) and
(R-1-2)--(R-2-2)--(R-3-2) corresponding respectively, to the variables
$\{c_n^m, \alpha_n^m, \gamma_n^m\}$ and $\{d_n^m, \beta_n^m,
\eta_n^m\}$. The coefficient matrices are denoted as $A_{n}$ and
$B_{n}$ in the following. By elementary linear algebra
manipulations, the augmented matrix for $A_n$ for the first system reduces to
\[\left(\begin{array}{cccc}h_n^{(1)}(2\omega) & 0 & j_n(2\omega) & f_{nm}^{(1)} \\0 & -\eps_0^{-1/2}j_n(k\omega) & \rho j_n(\omega\rho)-\frac{\rho j_n(2\omega)h_n^{(1)}(\omega\rho)}{h_n^{(1)}(2\omega)} & -\frac{f^{(1)}_{nm}\rho h_n^{(1)}(\omega\rho)}{h_n^{(1)}(2\omega)} \\0 & 0 & \tl A_{n}(3,3) & \tl A_{n}(3,4)\end{array}\right)\]
where
\[\begin{split}\tl A_{n}(3,3)=&\frac{\eps_0^{1/2}}{h_n^{(1)}(2\omega)j_n(k\omega)}\left\{\mu_0^{1/2}j_n(k\omega)
[\mathcal{J}_n(\omega\rho)h_n^{(1)}(2\omega)-\mathcal{H}_n(\omega\rho)j_n(2\omega)]\right.\\
&\left.-\rho\mu_0^{-1/2}\mathcal{J}_n(k\omega)[j_n(\omega\rho)h_n^{(1)}(2\omega)-h_n^{(1)}(\omega\rho)j_n(2\omega)]\right\},\end{split}\]
and
\[\tl A_{n}(3,4)=\frac{\eps_0^{1/2}f^{(1)}_{nm}}{h_n^{(1)}(2\omega)j_n(k\omega)}
\left\{\rho\mu_0^{-1/2}h_n^{(1)}(\omega\rho)\mathcal{J}_n(k\omega)-\mu_0^{1/2}\mathcal{H}_n(\omega\rho)j_n(k\omega)\right\}.\]
For $\mbox{det}(A_{n})=0$, one can choose $(\omega, k, \mu_0)$
satisfying
\begin{equation}\label{busting}
\mu_0\frac{j_n(k\omega)}{\mathcal{J}_n(k\omega)}
=\rho\frac{j_n(\omega\rho)h_n^{(1)}(2\omega)-h_n^{(1)}(\omega\rho)j_n(2\omega)}{\mathcal{J}_n(\omega\rho)h_n^{(1)}(2\omega)-\mathcal{H}_n(\omega\rho)j_n(2\omega)},
\end{equation}
It is easily verified that with this choice,
if $f_{nm}^{(1)}\neq0$, then
$\tl A_{n}(3,4)\neq 0$, there exists no solution of $(c_n^m,
\alpha_n^m, \gamma_n^m)$. %if $f_{nm}^{(1)}=0$, then $\tl A_{n}(3, 4)=0$, there are multiple solutions.
The boundary value problem is ill-posed and one does not have a well-defined boundary impedance map. In like manner,
one can find $(\omega, k, \eps_0)$ such that $\mbox{det}(B_{n})=0$. %For such medium $(\eps_0, \mu_0)$,
%frequency $\omega$ is known as resonance, while for such frequency $\omega$, medium $(\eps_0, \mu_0)$ in $B(0,1)$ is known as cloak-busting inclusion.
%We will show some numerical examples of such combinations in Section 6. In Section 5, to prevent the resonance $\omega$ for fixed medium, or equivalently, to prevent the cloak-busting inclusion $(\eps_0, \mu_0)$ for fixed frequency, we introduce the lossy
%approximate cloaking.
\\

%{\color{red}\begin{remark}Notice that from \eqref{t}, the denominator and the numerator of $t_2$ (hence the same for $t_1$)
%cannot be zero simultaneously for some $(\eps_0, \mu_0)$. Therefore, if we choose medium such that the denominator to be zero,
%to have a solution of (R-2-1) and (R-3-1), $\gamma_n^m$ has to be $0$. Then from (R-1-1) and (R-2-1) we can solve for $c_n^m$
%and $\alpha_n^m$ uniquely. In this case, the problem is well-posed, and corresponding $\mbox{det}(A_{n})\neq0$.
%\end{remark}}

Next we consider the performances of the approximate cloaking scheme
in extreme frequency regimes. That is, we let $\rho$ and $(\eps_0,
\mu_0)$ be fixed, and evaluate the approximate cloaking effects as
$\omega$ approaches zero or infinity, corresponding the low and high
frequency regimes. First, we see that
\[\nu\times H_\rho^+|_{\partial B_2}-\nu\times H|_{\partial B_2}=\sum_{n=1}^\infty\sum_{m=-n}^n g_{nm}^{(1)}U_n^m+g_{nm}^{(2)}V_n^m,\]
where \begin{equation}\label{bdry_b2}
\begin{split}g_{nm}^{(1)}:=&\frac{\omega}{i}\sqrt{n(n+1)}\left(b_n^mj_n(2\omega)-\eta_n^mj_n(2\omega)-d_n^mh_n^{(1)}(2\omega)\right)\\
%=&\displaystyle\frac{-4i\omega^2\sqrt{n(n+1)}f^{(2)}_{nm}t_3W_n(2\omega)}{t_3\mathcal{J}_n(2\omega)\mathcal{H}_n(2\omega)+\mathcal{J}_n(2\omega)^2}
=&\displaystyle\frac{-2i\sqrt{n(n+1)}\omega^2t_3b_n^mW_n(2\omega)}{t_3\mathcal{H}_n(2\omega)+\mathcal{J}_n(2\omega)},\\
=&\frac{-2i\sqrt{n(n+1)}\omega^2b_n^m\left[\eps_0\mathcal{J}_n(\omega\rho)j_n(k\omega)-\rho j_n(\omega\rho)\mathcal{J}_n(k\omega)\right]W_n(2\omega)}{\eps_0^{1/2}\mbox{det}(B_{n})},\\
g_{nm}^{(2)}:=&\frac{1}{2i\omega}\sqrt{n(n+1)}\left(a_n^m\mathcal{J}_n(2\omega)-\gamma_n^m\mathcal{J}_n(2\omega)-c_n^m\mathcal{H}_n(2\omega)\right)\\
%=&\displaystyle\frac{i\sqrt{n(n+1)}t_1f^{(1)}_{nm}W_n(2\omega)}{t_1j_n(2\omega)h^{(1)}_n(2\omega)+j_n(2\omega)^2}
=&\displaystyle\frac{i\sqrt{n(n+1)}t_1a_n^mW_n(2\omega)}{t_1h_n^{(1)}(2\omega)+j_n(2\omega)}\\
=&\frac{i\sqrt{n(n+1)}a_n^m\left[\mu_0\mathcal{J}_n(\omega\rho)j_n(k\omega)-\rho j_n(\omega\rho)\mathcal{J}_n(k\omega)\right]W_n(2\omega)}{\mu_0^{1/2}\mbox{det}(A_{n})}
\end{split}\end{equation}
with $W_n(t):=j_n(t){h_n^{(1)}}'(t)-h_n^{(1)}(t)j_n'(t)$.

We shall address the frequency dependence issue by assuming that the
inputs are given by the EM plane waves of the form
(\ref{eq:planewave}). The corresponding coefficients $f_{nm}^{(1)}$
and $f^{(2)}_{nm}$ are given by \eqref{eq:f_planewave}, while
$a_n^m$ and $b_n^m$ are given by \eqref{eq:planewave_coeff}. It is
readily seen
\[a_n^m=\mO(1),\qquad b_n^m=\mO(\omega^{-1}).\]
For the low frequency regime with $\omega\ll1$, by \eqref{bdry_b2},
it is straightforward to show
\[g_{nm}^{(1)}\sim\omega^n\rho^{2n+1},\quad g_{nm}^{(2)}\sim \omega^{n-1}\rho^{2n+1},\]
which implies a satisfactory approximate cloaking. Whereas for the
high frequency regime with $\omega\gg 1$, we exclude the influence
of resonances from our study by considering the case that
$|\mbox{det}(A_{n})|$ and $|\mbox{det}(B_{n})|$ are bounded from
below by a positive function $C_{nm}(\omega, \rho)$, where the
transmission problem \eqref{Max_apprx_3} is well-posed. Then we
consider two separate cases:
\begin{itemize}
\item When $1\leq\omega\ll\rho^{-1}$, i.e., $\omega\rho\ll1$, since
$j_n(t), h_n^{(1)}(t)$ oscillate between $-t^{-1}$ and $t^{-1}$,
$\mathcal{J}_n(t), \mathcal{H}_n(t)$ oscillate between $-1$ and $1$,
and $W_n(t)\sim t^{-2}$ as $t$ increases, we have
\begin{equation}\label{eq:r1}
|g_{nm}^{(1)}|\lesssim
\frac{(\omega\rho)^n\omega^{-2}}{C_{nm}(\omega, \rho)},\quad
|g_{nm}^{(2)}|\lesssim
\frac{(\omega\rho)^n\omega^{-3}}{C_{nm}(\omega, \rho)},
\end{equation}
where one can show that $C_{nm}(\omega, \rho)\lesssim
\omega^{-n-3}\rho^{-n-1}$. Here and in the following, for two
expressions $\mathcal{R}_1$ and $\mathcal{R}_2$, by
``$\mathcal{R}_1\lesssim\mathcal{R}_2$" we mean ``$\mathcal{R}_1\leq
c\mathcal{R}_2$" with a constant $c$.

\item For even higher frequency $\omega\gg1$
such that $\omega\rho\gtrsim 1$, we calculate
\begin{equation}\label{eq:r2}
|g_{nm}^{(1)}|\lesssim \frac{\omega^{-2}}{C_{nm}(\omega,\rho)},\quad
|g_{nm}^{(2)}|\lesssim \frac{\omega^{-3}}{C_{nm}(\omega,\rho)}
\end{equation}
where $C_{nm}\lesssim \omega^{-2}$.
\end{itemize}
By (\ref{eq:r1}) and (\ref{eq:r2}), we cannot conclude whether or
not one can achieve the near-invisibility. However, we conducted
extensive numerical experiments to see that one has the
near-invisibility in these two cases. These suggest that for the
cloaking of passive media, excluding the resonance frequencies, one
could achieve the near-invisibility for the approximate cloaking
scheme of every frequency. This is sharply different from the case
of cloaking active/radiating objects which we shall consider in the
next section. }

So far, we have been concerned with the cloaking device where the
cloaked region is $B_1$ and the cloaking medium occupies
$B_2\backslash\bar{B}_1$, which we obtain by using the
transformation (\ref{eq:coefficients})--(\ref{eq:transf}). It is
remarked here that for arbitrary $0<R_1<R_2<\infty$, one can
construct an approximate cloaking device whose cloaked region is
$B_{R_1}$ and the cloaking layer is $B_{R_2}\backslash\bar{B}_1$ by
implementing the following transformation
\[
x:=G_\rho(y)=\left\{\begin{array}{ll}G_\rho^{(1)}(y)=(a+b|y|)\frac{y}{|y|}
&\rho<|y|<R_2,\\
G_\rho^{(2)}(y)=\frac{y}{\rho}&|y|\leq\rho,\end{array}\right.
\]
where
\[
a=\frac{R_1-\rho}{R_2-\rho}R_2,\quad b=\frac{R_2-R_1}{R_2-\rho}.
\]
It is readily seen that all our earlier results hold for such
construction. The remark applies equally to all our subsequent
study.

\section{Approximate cloaking with an internal electric current at origin}

In this section, we consider the approximate EM cloaking scheme
constructed in Section 3.1 in the case that we have an internal electric current
present in the cloaked region supported at the origin. The corresponding EM fields verify
\begin{equation}\label{source_1}
\left\{\begin{array}{l}
\nabla\times\tl{E}_\rho=i\omega\tl{\mu}_\rho\tl{H}_\rho, \;\;\;\;\;\nabla\times \tl{H}_\rho=-i\omega\tl{\eps}_\rho\tl{E}_\rho+\tl{J},\;\;\;\;\;\mbox{in }\;B_2\\
\nu\times\tl{E}_\rho|_{\partial B_2}=f,
\end{array}\right.
\end{equation}
where $\tilde J$ has the form
\begin{equation}\label{source_3}
\tl{J}=\sum_{|\alpha|<K}(\partial_x^\alpha\delta_0(x))\mathbf{v}_\alpha,
\end{equation}
with $\delta_0$ denoting the Dirac delta function at origin and
$\mathbf{v}_\alpha\in \C^3$. The pull-back EM fields satisfy
\begin{equation}\label{source_2}
\left\{\begin{array}{l}\nabla\times E_\rho=i\omega\mu_\rho H_\rho, \;\;\;\;\; \nabla\times H_\rho=-i\omega\eps_\rho E_\rho+J, \;\;\;\;\;\mbox{in }\;B_2,\\
\nu\times E_\rho|_{\partial B_2}=f,\end{array}\right.
\end{equation}
where $J=(F_\rho^{2})^*\tilde{J}$.

The point electric current $\tl{J}$ would give rise to a radiating
field
\begin{equation}
E_{\tl{J}}=\displaystyle\sum_{n=1}^K\sum_{m=-n}^np_n^mN_{n, k\omega}^m+q_n^m\nabla\times N_{n, k\omega}^m.
\end{equation}
Hence for $x\in B_1$
\begin{equation}\label{E_src}
\tl
E^-_\rho=\eps_0^{-1/2}\sum_{n=1}^{\infty}\sum_{m=-n}^{n}\alpha_n^mM_{n,k\omega}^m+\beta_n^m\nabla\times
M_{n,k\omega}^m+p_n^mN_{n, k\omega}^m+q_n^m\nabla\times N_{n,
k\omega}^m,
%\\ \displaystyle\tl H^-_\rho=\frac{1}{ik\omega}\mu_0^{-1/2}\sum_{n=1}^{\infty}\sum_{m=-n}^{n}
%k^2\omega^2\beta_n^mM_{n,k\omega}^m+\alpha_n^m\nabla\times M_{n,k\omega}^m+k^2\omega^2q_n^mN_{n, k\omega}^m+p_n^m\nabla\times N_{n,k\omega}^m,
\end{equation}
where $p_n^m$ and $q_n^m$ equal zero when $n>K$. Whereas $E_\rho^+$
and $H_\rho^+$ are as in (\ref{plbk}).
\begin{lemma}\label{lem:estimate source}
Assume $\omega$ is not an eigenvalue to \eqref{source_1}. Let $(\tl
E_\rho, \tl H_\rho)$ be the EM fields to \eqref{source_1} with $\tl
J$ given by \eqref{source_3}, and $(E_\rho, H_\rho)=((F_\rho)^*\tl
E_\rho, (F_\rho)^*\tl H_\rho)$ be the EM fields to \eqref{source_2}.
Given $\tl E_\rho^-$ as in \eqref{E_src} and $E_\rho^+$ as in
\eqref{plbk}, we have as $\rho\rightarrow 0^+$,
\begin{equation}\label{ga_et_c_d_src}
\gamma_n^m=\mO(1),\;\;\;
\eta_n^m=\mO(1);\;\;\;c_n^m=\mO(\rho^{n+1}),\;\;\;d_n^m=\mO(\rho^{n+1}),
\end{equation} and
\begin{equation}\label{alpha_beta_src}\alpha_n^m=\mO(1),\;\;\; \beta_n^m=\mO(1).
\end{equation}
\end{lemma}
\begin{proof}
The boundary condition on $\partial B_2$ implies (R-1). From the
standard transmission conditions, we have
 \[\mbox{(R-2')}\left\{\begin{array}{l}\rho c_n^mh_n^{(1)}(\omega\rho)+\rho\gamma_n^mj_n(\omega\rho)=\eps_0^{-1/2}(\alpha_n^m j_n(k\omega)+p_n^m h_n^{(1)}(k\omega)),\\
d_n^m\mathcal{H}_n(\omega\rho)+\eta_n^m\mathcal{J}_n(\omega\rho)=\eps_0^{-1/2}(\beta_n^m\mathcal{J}_n(k\omega)+q_n^m\mathcal{H}_n(k\omega)),
\end{array}\right.\]
and
\[\mbox{(R-3')}\left\{\begin{array}{l}k c_n^m\mathcal{H}_n(\omega\rho)+k\gamma_n^m\mathcal{J}_n(\omega\rho)=\mu_0^{-1/2}(\alpha_n^m\mathcal{J}_n(k\omega)+p_n^m\mathcal{H}_n(k\omega)),\\
\rho
d_n^mh_n^{(1)}(\omega\rho)+\rho\eta_n^mj_n(\omega\rho)=\mu_0^{-1/2}(k\beta_n^m
j_n(k\omega)+k q_n^mh_n^{(1)}(k\omega)).
\end{array}\right.\]
Solving (R-2') and (R-3'), we obtain
\begin{equation}\label{al_be_c_d_Src}
\begin{array}{ll}c_n^m=t_1\gamma_n^m+t_1'p_n^m,&\alpha_n^m=t_2\gamma_n^m+t_2'p_n^m,\\
d_n^m=t_3\eta_n^m+ t_3'q_n^m,&\beta_n^m=t_4\eta_n^m+
t_4'q_n^m,\end{array}\end{equation} where $t_i$ $(i=1,2,3,4)$ are
given by \eqref{t} and $t'_i$ $(i=1,2,3,4)$ are given by
\begin{equation}\label{t'}
\begin{array}{l}
t_1'=\displaystyle\frac{h^{(1)}_n(k\omega )\mathcal{J}_n(k\omega )-\mathcal{H}_n(k\omega )j_n(k\omega )}{\mu_0^{-1/2}\rho h_n^{(1)}(\omega\rho)\mathcal{J}_n(k\omega )-\eps_0^{-1/2}k\mathcal{H}_n(\omega\rho)j_n(k\omega )}=\mO(\rho^{n+1}),\\
t_2'=\displaystyle\frac{\eps_0^{-1/2}k h^{(1)}_n(k\omega )\mathcal{H}_n(\omega\rho)-\mu_0^{-1/2}\rho\mathcal{H}_n(k\omega )h^{(1)}_n(\omega\rho)}{\mu_0^{-1/2}\rho h_n^{(1)}(\omega\rho)\mathcal{J}_n(k\omega )-\eps_0^{-1/2}k\mathcal{H}_n(\omega\rho)j_n(k\omega )}=\mO(1),\\
t_3':=\displaystyle\frac{\mathcal{J}_n(k\omega)h^{(1)}_n(k\omega)-\mathcal{H}_n(k\omega)j_n(k\omega)}{\eps_0^{-1/2}\rho h_n^{(1)}(\omega\rho)\mathcal{J}_n(k\omega)-\mu_0^{-1/2}k\mathcal{H}_n(\omega\rho)j_n(k\omega)}=\mO(\rho^{n+1}),\\
t_4':=\displaystyle\frac{\mu_0^{-1/2}k
h^{(1)}_n(k\omega)\mathcal{H}_n(\omega\rho)-\eps_0^{-1/2}\rho\mathcal{H}_n(k\omega)h^{(1)}_n(\omega\rho)}{\eps_0^{-1/2}\rho
h_n^{(1)}(\omega\rho)\mathcal{J}_n(k\omega)-\mu_0^{-1/2}k\mathcal{H}_n(\omega\rho)j_n(k\omega)}=\mO(1).
\end{array}
\end{equation}
Plugging into (R-1), we obtain
\begin{equation}\label{ga_et_src}
\gamma_n^m=\displaystyle\frac{f_{nm}^{(1)}-p_n^m
t_1'h_n^{(1)}(2\omega)}{t_1h_n^{(1)}(2\omega)+j_n(2\omega)}=\mO(1),\;\;\;
\eta_n^m=\displaystyle\frac{2f_{nm}^{(2)}- t_3'
q_n^m\mathcal{H}_n(2\omega)}{t_3\mathcal{H}_n(2\omega)+\mathcal{J}_n(2\omega)}=\mO(1),\end{equation}
which together with \eqref{al_be_c_d_Src} imply
\eqref{ga_et_c_d_src} and \eqref{alpha_beta_src}.
\end{proof}

Next, we evaluate the performances of the approximate EM cloaking.

\begin{proposition}\label{prop:approximate cloaking 2}
For the approximate EM cloaking with an internal point
current~(\ref{source_3}) present in the cloaked region, if
$\omega$ is not an eigenvalue to \eqref{source_1}, we have
\begin{equation}\label{eq:boundary norm estimate 2}
\|\Lambda_{\tilde\varepsilon_\rho,\tilde\mu_\rho,\tilde J}^\omega-
\Lambda_0^\omega\|=\mathcal{O}(\rho^2)\quad\mbox{as\
$\rho\rightarrow 0^+$},
\end{equation}
where $\|\cdot\|$ denotes the operator norm of the impedance map.
\end{proposition}

\begin{proof}%\label{rk_7}
On the boundary $\partial B_2$, using the expression (\ref{bdry_2})
for $\nu\times H_\rho^+|_{\partial B_2}$ and (\ref{eq:ideal H field}) for $\nu\times
H|_{\partial B_2}$, together with the asymptotic estimates of the corresponding
coefficients in Lemma~\ref{lem:estimate source}, we have
(\ref{eq:boundary norm estimate 2}) by straightforward comparisons,
since the coefficients $c_n^m$, $d_n^m$, $\gamma_n^m-a_n^m$ and
$\eta_n^m-b_n^m$ converge to zero of order $\mO(\rho^{n+1})$.
\end{proof}

By Proposition~\ref{prop:approximate cloaking 2}, we see that one
still achieves near-invisibility cloaking even though there is a
source/sink present in the cloaked region. That is, the
approximate cloaking makes both the passive medium and the active
point source/sink nearly-invisible. However, we have one order
reduction of the convergence rate. {This is due to the extra terms
\[\frac{-p_n^mt_1'h_n^{(1)}(2\omega)}{t_1h_n^{(1)}(2\omega)-j_n(2\omega)},\;\;
\frac{-q_n^mt_3'\mathcal{H}_n(2\omega)}{t_3\mathcal{H}_n(2\omega)-\mathcal{J}_n(2\omega)},\;\;
t_1'p_n^m,\;\; t_3'q_n^m\,\sim\rho^{n+1}\] in $\gamma_n^m-a_n^m$,
$\eta_n^m-b_n^m$, $c_n^m$ and $d_n^m$ respectively, compared to the
case without the source/sink.}

Next, we consider the limiting status of the approximate cloaking in
this case when a point source/sink is present. We have
\begin{proposition}\label{prop:source limiting}
Assume $\omega$ is not an eigenvalue to \eqref{source_1}. Let $(\tl
E_\rho, \tl H_\rho)$ be the EM fields satisfying \eqref{source_1} and
$(E_\rho, H_\rho)=((F_\rho)^*\tl E_\rho, (F_\rho)^*\tl H_\rho)$ be
the EM fields satisfying \eqref{source_2}. Then we have as $\rho\rightarrow
0^+$,
\begin{equation}\label{eq:source convergence 1}
(E_\rho^+, H_\rho^+)\rightarrow (E, H)
\end{equation}
with $(E,H)$ being the EM fields on $B_2$ in the free space. Also
\begin{equation}\label{eq:source convergence 2}
(\tilde{E}_\rho^{-}, \tilde{H}_\rho^{-})\rightarrow (\hat{E}^{-},
\hat{H}^{-}),
\end{equation}
where $(\hat{E}^{-}, \hat{H}^{-})$ satisfy the Maxwell equations
\begin{equation}\label{eq:source convergence 3}
\nabla\times\hat{E}^-=i\omega\mu_0\hat{H}^{-},\quad\nabla\times\hat{H}^-=-i\omega\varepsilon_0\hat{E}^-+\tilde
J \quad\mbox{in \ $B_1$}
\end{equation}
with
\begin{equation}\label{eq:source convergence 4}
\nu\times\hat{E}^-|_{\Sigma^-}\neq 0\quad\mbox{and}\quad
\nu\times\hat{H}^-|_{\Sigma^-}\neq 0.
\end{equation}
\end{proposition}

\begin{proof}
{By a similar argument to the first part of the proof of
Proposition \ref{prop:convergence}, one can show that on any compact
subset of $B_2$ away from the origin,
$(E_\rho^+,H_\rho^+)\rightarrow (E, H)$ at the rate
$\mathcal{O}(\rho^2)$, and on $B_2\backslash\overline{B_\rho}$,
\[
\|E_\rho^+-E\|_{L^2(B_2\backslash\overline{B_\rho})}+
\|H_\rho^+-H\|_{L^2(B_2\backslash\overline{B_\rho})}=\mathcal{O}(\rho^{1/2})\quad\mbox{as\
\ $\rho\rightarrow 0^+$}.
\]
%\[\int_{B_2\backslash\overline{B_\rho}}|E_\rho^+-E|^2+|H_\rho^+-H|^2dx=\mO(\rho),\] which
This proves \eqref{eq:source convergence 1}.} Next, we shall show
(\ref{eq:source convergence 4}) which in turn implies
(\ref{eq:source convergence 2})--(\ref{eq:source convergence 3}). On
the interior cloaking interface $\Sigma^-$, the Cauchy data are
given by {\begin{equation}\label{eq:source proof 1}
\left\{\begin{array}{l}
\nu\times\tl E_\rho^-|_{\Sigma^-}=\displaystyle\eps_0^{-1/2}\sum_{n=1}^\infty\sum_{m=-n}^n\sqrt{n(n+1)}\Big((\alpha_n^mj_n(k\omega)+p_n^mh_n^{(1)}(k\omega))U_n^m\\
\hspace{5cm}+(\beta_n^m\mathcal{J}_n(k\omega)+q_n^m\mathcal{H}_n(k\omega))V_n^m\Big),\\
\nu\times\tl H_\rho^-|_{\Sigma^-}=\displaystyle\frac{\mu_0^{-1/2}}{ik\omega}\sum_{n=1}^\infty\sum_{m=-n}^n
\sqrt{n(n+1)}\Big((\alpha_n^m\mathcal{J}_n(k\omega)+p_n^m\mathcal{H}_n(k\omega))V_n^m\\
\hspace{5cm}+k^2\omega^2(\beta_n^mj_n(k\omega )+q_n^m h_n^{(1)}(k\omega))U_n^m\Big).
\end{array}\right.
\end{equation}}
We observe that as $\rho\rightarrow 0^+$
\begin{equation}\label{eq:source proof 2}
\begin{split}
\alpha_n^mj_n(k\omega)+p_n^mh_n^{(1)}(k\omega)&=t_2\gamma_n^mj_n(k\omega)+(t_2'j_n(k\omega)+h_n^{(1)}(k\omega))p_n^m=\mO(\rho),\\
\beta_n^m\mathcal{J}_n(k\omega)+q_n^m\mathcal{H}_n(k\omega)&=t_4\eta_n^m\mathcal{J}_n(k\omega)+(t_4'\mathcal{J}_n(k\omega)+\mathcal{H}_n(k\omega))q_n^m=\mO(1),
\end{split}
\end{equation}
where
\[\begin{split}&t_2'j_n(k\omega)+h_n^{(1)}(k\omega)\sim \frac{\rho h_n^{(1)}(\omega\rho)[j_n(k\omega)\mathcal{H}_n(k\omega)-h_n^{(1)}(k\omega)\mathcal{J}_n(k\omega)]}{\mu_0\mathcal{H}_n(\omega\rho)j_n(k\omega)}=\mO(\rho),\\
&t_4'\mathcal{J}_n(k\omega)+\mathcal{H}_n(k\omega)\sim
\frac{j_n(k\omega)\mathcal{H}_n(k\omega)-\mathcal{J}_n(k\omega)h_n^{(1)}(k\omega)}{j_n(k\omega)}=\mO(1).
\end{split}\]
Similarly, we have
\begin{equation}\label{eq:source proof 3}
\begin{split}
&\beta_n^mj_n(k\omega)+q_n^mh_n^{(1)}(k\omega)=\mO(\rho),\\
&\alpha_n^m\mathcal{J}_n(k\omega)+p_n^m\mathcal{H}_n(k\omega)=\mO(1).
\end{split}
\end{equation}
Plugging (\ref{eq:source proof 2}) and (\ref{eq:source proof 3})
into (\ref{eq:source proof 1}), we have (\ref{eq:source convergence
4}). The proof is completed.
\end{proof}

By Proposition~\ref{prop:source limiting}, we see that as
$\rho\rightarrow 0^+$, the near-cloak converges to the ideal-cloak.
Moreover, in the limiting case, the EM fields in the cloaked region
are trapped inside and the cloaked region is completely isolated.
%\begin{remark}\label{rk_4}
%When an internal current source is present at the origin, the EM
%fields $(\tl E^-_\rho, \tl H^-_\rho)$ in the cloaked region is no
%longer vanishing as $\rho\rightarrow0$. Furthermore, comparing
%\eqref{al_be_c_d_Src} with \eqref{al_be_c_d_nSrc} in the case
%without the source, we observe extra $t_2'p_n^m$ and $t_4'q_n^m$ (of
%$\mO(1)$) in the coefficients $\alpha_n^m$ and $\beta_n^m$ of entire
%field, i.e., the terms associated to $M_{n,k\omega}^m$ and
%$\nabla\times M_{n,k\omega}^m$.
%\end{remark}
\\

{\color{black} Finally, we consider the frequency dependence for the
approximate cloaking of active/radiating objects. Again, we address
our study by considering the inputs being EM plane waves as in
Section 3.3. By straightforward calculations, the coefficients that
characterize the difference of the boundary measurements, i.e.
$\nu\times H_\rho^+|_{\partial B_2}-\nu\times H|_{\partial B_2}$,
associated to terms $U_n^m$ and $V_n^m$, verify
{\begin{equation}\label{freq_lossless}
\begin{split}\tilde{g}_{nm}^{(1)}:=&\frac{\omega}{i}\sqrt{n(n+1)}(b_n^mj_n(2\omega)-\eta_n^mj_n(2\omega)-d_n^mh_n^{(1)}(2\omega))\\
%=&\frac{-2i\omega^2\sqrt{n(n+1)}[2f^{(2)}_{nm}t_3+q_n^mt_3'\mathcal{J}_n(2\omega)]
%W_n(2\omega)}{t_3\mathcal{J}_n(2\omega)\mathcal{H}_n(2\omega)+\mathcal{J}_n(2\omega)^2}\\
%=&\frac{-2i\sqrt{n(n+1)}\omega^2(t_3b_n^m+t_3'q_n^m)W_n(2\omega)}{t_3\mathcal{J}_n(2\omega)+\mathcal{H}_n(2\omega)}\\
%=&\frac{-2i\sqrt{n(n+1)}\omega^2W_n(2\omega)}{\eps_0^{1/2}\mbox{det}(B_{n})}
%\Big[b_n^m\left(\eps_0\mathcal{J}_n(\omega\rho)j_n(k\omega)-\rho j_n(\omega\rho)\mathcal{J}_n(k\omega)\right)\\
%&+\eps_0^{1/2}q_n^m\left(\mathcal{J}_n(k\omega)h_n^{(1)}(k\omega)-\mathcal{H}_n(k\omega)j_n(k\omega)\right)\Big]\\
=&\frac{-2i\sqrt{n(n+1)}\omega^2W_n(2\omega)}{\eps_0^{1/2}\mbox{det}(B_{n})}\Big[b_n^m\left(\eps_0\mathcal{J}_n(\omega\rho)j_n(k\omega)-\rho j_n(\omega\rho)\mathcal{J}_n(k\omega)\right)\\
&-\eps_0^{1/2}q_n^m k\omega W_n(k\omega)\Big],\\
\tilde{g}_{nm}^{(2)}:=&\frac{1}{2i\omega}\sqrt{n(n+1)}(a_n^m\mathcal{J}_n(2\omega)-\gamma_n^m\mathcal{J}_n(2\omega)-c_n^m\mathcal{H}_n(2\omega))\\
%=&\frac{i\sqrt{n(n+1)}[t_1f^{(1)}_{nm}+p_n^mt_1'j_n(2\omega)]W_n(2\omega)}{t_1j_n(2\omega)h^{(1)}_n(2\omega)+j_n(2\omega)^2}\\
%=&\frac{i\sqrt{n(n+1)}(t_1a_n^m+t_1'p_n^m)W_n(2\omega)}{t_1h_n^{(1)}(2\omega)+j_n(2\omega)}\\
%=&\frac{i\sqrt{n(n+1)}W_n(2\omega)}{\mu_0^{1/2}\mbox{det}(A_{n})}\Big[a_n^m\left(\mu_0\mathcal{J}_n(\omega\rho)j_n(k\omega)
%-\rho j_n(\omega\rho)\mathcal{J}_n(k\omega)\right)\\
%&+\mu_0^{1/2}p_n^m\left(h_n^{(1)}(k\omega)\mathcal{J}_n(k\omega)-\mathcal{H}_n(k\omega)j_n(k\omega)\right)\Big]\\
=&\frac{i\sqrt{n(n+1)}W_n(2\omega)}{\mu_0^{1/2}\mbox{det}(A_{n})}\Big[a_n^m\left(\mu_0\mathcal{J}_n(\omega\rho)j_n(k\omega)-\rho j_n(\omega\rho)\mathcal{J}_n(k\omega)\right)\\
&-\mu_0^{1/2}p_n^mk\omega W_n(k\omega)\Big].
\end{split}
\end{equation}}
%From \eqref{t'}, we have
%\[\begin{split}t_1'&=\displaystyle\frac{-k\omega W_n(k\omega)}{\mu_0^{-1/2}\rho h_n^{(1)}(\omega\rho)\mathcal{J}_n(k\omega )-\eps_0^{-1/2}k\mathcal{H}_n(\omega\rho)j_n(k\omega )}=\mO(\rho^{n+1}),\\
%t_3'&=\displaystyle\frac{-k\omega W_n(k\omega)}{\eps_0^{-1/2}\rho
%h_n^{(1)}(\omega\rho)\mathcal{J}_n(k\omega)-\mu_0^{-1/2}k\mathcal{H}_n(\omega\rho)j_n(k\omega)}=\mO(\rho^{n+1}).\end{split}\]
In the low frequency regime with $\omega\ll1$, by
\eqref{freq_lossless} we have
\[\tilde{g}_{nm}^{(1)}\sim \omega^{-n}\rho^{n+1},\quad \tilde{g}_{nm}^{(2)}\sim \omega^{-n-2}\rho^{n+1},\]
which implies that one cannot achieve near-invisibility when
$\omega\lesssim\rho^{2/3}$. In the high frequency regime with
$\omega\gg 1$, by excluding the resonances and using similar
arguments to that in Section 3.3, one can show
\begin{equation}\label{eq:r3}
|\tilde{g}_{nm}^{(1)}|\lesssim
\frac{\omega^{-1}}{C_{nm}(\omega,\rho)}, \quad
|\tilde{g}_{nm}^{(2)}|\lesssim
\frac{\omega^{-3}}{C_{nm}(\omega,\rho)},
\end{equation}
where
\[C_{nm}(\omega,\rho)\lesssim\left\{\begin{array}{ll}\omega^{-n-3}\rho^{-n-1}&\omega\rho\ll1,\,\omega\gtrsim1,\\ \omega^{-2}&\omega\rho\gtrsim1.\end{array}\right. \]
By (\ref{eq:r3}), one cannot conclude whether or not the
near-invisibility is achieved. However, in our numerical experiment
given in Section 6.3, we have observed the failure of the
approximate cloaking in the high frequency regime. Therefore, it can
be concluded that for a fixed approximate cloaking scheme with a
point source/sink (\ref{source_3}) present in the cloaked region, in
addition to resonances, the near-invisibility cannot be achieved
uniformly in frequency.

}

\section{Approximate cloaking with a lossy layer}

In our earlier discussion of lossless approximate cloakings, we have
seen the failure of the near-invisibility due to resonant
inclusions. Following the spirit in \cite{KOVW} by introducing a
damping mechanism to overcome resonances in approximate acoustic
cloaking, we surround the cloaked region first by an isotropic
conducting layer, then another anisotropic nonconducting layer as
described as earlier. To be more specific, given a damping parameter
$\tau>0$, our new regularized parameter in $B_2$ is given by
\begin{equation}\label{lossy_coeff}
(\tl\mu_\rho(x),
\tl\eps_\rho(x))=\left\{\begin{array}{ll}((F_{2\rho})_*I,
(F_{2\rho})_*I)\;\;&1<|x|<2,\\(\mu_\tau,
\eps_\tau):=((F_{2\rho})_*I,
(F_{2\rho})_*(1+i\tau))\;\;&\frac{1}{2}<|x|<1,\\ (\mu_0,
\eps_0)&|x|<\frac{1}{2},\end{array}\right.\end{equation} which is
the push-forward of
\begin{equation}\label{lossy_coeff_virt}(\mu_\rho(y), \eps_\rho(y))=\left\{\begin{array}{ll}(I, I)\;\;&2\rho<|y|<2,\\(I, 1+i\tau)\;\;&\rho<|y|<2\rho,\\((F_{2\rho}^{-1})_*\mu_0, (F_{2\rho}^{-1})_*\eps_0)\;\;&|y|<\rho.\end{array}\right.\end{equation}
by the transformation
\[x:=F_{2\rho}(y)=\left\{\begin{array}{ll}(\frac{1-2\rho}{1-\rho}+\frac{1}{2(1-\rho)}|y|)\frac{y}{|y|}&2\rho<|y|<2,\\ \frac{y}{2\rho}&|y|\leq2\rho.\end{array}\right.\]

%The EM fields we are interested in are $(E_1, H_1)$ in virtual space $B_2\backslash\overline{B_{2\rho}}$,
%$(\tl E_2, \tl H_2)$ in physical space $B_1\backslash\overline{B_{1/2}}$ and $(\tl E_3, \tl H_3)$ in physical space $B_{1/2}$. They satisfy the following problem
To assess the approximate cloaking in this setting, we consider the
transmission problem
\begin{equation}\label{damp}
\left\{
\begin{array}{l}\nabla\times E_1=i\omg H_1, \;\;\;\;\;\nabla\times H_1=-i\omg E_1,\;\;\;\;\;\mbox{in }\;2\rho<|y|<2,\\
\nabla\times\tl{E}_2=i\omg\mu_\tau\tl{H}_2,\;\;\;\;\;\nabla\times \tl{H}_2=-i\omg\eps_\tau\tl{E}_2,\;\;\;\;\;\mbox{in }\;\frac{1}{2}<|x|<1,\\
\nabla\times\tl{E}_3=i\omg\mu_0\tl{H}_3,\;\;\;\;\;\nabla\times \tl{H}_3=-i\omg\eps_0\tl{E}_3+\tl J,\;\;\;\;\;\mbox{in }\;|x|<\frac{1}{2},\\
\nu\times E_1|_{\partial B_2}=f;\\
\nu\times\tl{E}_2|_{\partial B_1^-}=2\rho(\nu\times E_1)|_{\partial
B_{2\rho}^+},\;\;\;
\nu\times\tl{H}_2|_{\partial B_1^-}=2\rho(\nu\times H_1)|_{\partial B_{2\rho}^+};\\
\nu\times\tl{E}_3|_{\partial B_{1/2}^-}=\nu\times
\tl{E}_2|_{\partial B_{1/2}^+},\;\;\;\nu\times\tl{H}_3|_{\partial
B_{1/2}^-}=\nu\times \tl{H}_2|_{\partial B_{1/2}^+}.
\end{array}\right.
\end{equation}
The problem is well-posed on $B_2$ since $\eps_\tau$ is complex.
Actually, we have
\[(\mu_\tau, \eps_\tau)=(2\rho, 2\rho(1+i\tau)).\]
Set
\[k_\tau:=(\mu_\tau\eps_\tau)^{1/2}=\mO(\rho)\quad\mbox{as }\;\rho\rightarrow0^+.\]
We can write the spherical wave expansions of the electric fields as
follows
\begin{equation}\label{lossy_expa}\left\{\begin{array}{l}
E_1=\displaystyle\sum_{n=1}^\infty\sum_{m=-n}^n\gamma_n^mM_{n,\omega}^m+\eta_n^m\nabla\times M_{n,\omega}^m+c_n^m N_{n,\omega}^m+d_n^m\nabla\times N_{n,\omega}^m, \\
\tl{E}_2=\displaystyle\eps_\tau^{-1/2}\sum_{n=1}^\infty\sum_{m=-n}^n\tl{\gamma}_n^mM_{n, k_\tau\omega}^m+\tl\eta_n^m\nabla\times M_{n, k_\tau\omega}^m+\tl{c}_n^m N_{n, k_\tau\omega}^m+\tl{d}_n^m\nabla\times N_{n, k_\tau\omega}^m,\\
\tl{E}_3=\eps_0^{-1/2}\displaystyle\sum_{n=1}^\infty\sum_{m=-n}^n\alpha_n^mM_{n,
k\omega}^m+\beta_n^m\nabla\times M_{n, k\omega}^m+p_n^mN_{n,
k\omega}^m+q_n^m\nabla\times N_{n, k\omega}^m.
\end{array}\right.\end{equation}
Then we have
\begin{proposition}\label{prop_lossy}
For any $\omega\in\R^+$, assume the EM field $(\tl E_\rho, \tl H_\rho)$
satisfies
\[
\left\{\begin{array}{l}
\nabla\times\tl{E}_\rho=i\omega\tl{\mu}_\rho\tl{H}_\rho, \;\;\;\;\;\nabla\times \tl{H}_\rho=-i\omg\tl{\eps}_\rho\tl{E}_\rho+\tl J,\;\;\;\;\;\mbox{in }\;B_2\\
\nu\times\tl{E}_\rho|_{\partial B_2}=f,
\end{array}\right.
\] where $(\tl\mu_\rho, \tl\eps_\rho)$ is the lossy medium given by \eqref{lossy_coeff}. Then the pull-back field $(E_\rho, H_\rho)$ satisfies
\[\left\{\begin{array}{l}\nabla\times E_\rho=i\omg\mu_\rho H_\rho, \;\;\;\;\; \nabla\times H_\rho=-i\omg\eps_\rho E_\rho+J, \;\;\;\;\;\mbox{in }\;B_2,\\
\nu\times E_\rho|_{\partial B_2}=f\end{array}\right.\] with
$(\mu_\rho, \eps_\rho)$ given by \eqref{lossy_coeff_virt}.
Therefore, the fields
\[\begin{split}(E_1, H_1)&=(E_\rho|_{B_2\backslash \overline{B_{2\rho}}}, H_\rho|_{B_2\backslash \overline{B_{2\rho}}}),\\
(\tl E_2, \tl H_2)&=(\tl E_\rho|_{B_1\backslash\overline{B_{1/2}}}, \tl H_\rho|_{B_1\backslash\overline{B_{1/2}}}),\\
(\tl E_3, \tl H_3)&=(\tl E_\rho|_{B_{1/2}}, \tl
H_\rho|_{B_{1/2}})\end{split}\] satisfy the transmission problem
\eqref{damp}. Moreover,
\begin{enumerate}[\upshape (i)]
\item If $\tl J=0$, then $(E_1, \tl E_2, \tl E_3)$ is given by \eqref{lossy_expa} with $p_n^m=q_n^m=0$ for all $n$ and $m$, and
\begin{equation}\label{lossy_est}\left\{\begin{array}{l}\gamma_n^m=\mO(1), \;\;\eta_n^m=\mO(1),\;\;c_n^m=\mO(\rho^{2n+1}), \;\; d_n^m=\mO(\rho^{2n+1});\\
\tl c_n^m=\mO(\rho^{2n+5/2}), \;\;\tl d_n^m=\mO(\rho^{2n+5/2}),\;\;\tl\gamma_n^m=\mO(\rho^{3/2}), \;\;\tl\eta_n^m=\mO(\rho^{3/2});\\
\alpha_n^m=\mO(\rho^{n+1}),\;\;\beta_n^m=\mO(\rho^{n+1}).\end{array}\right.\end{equation}
\item If $\tl J\neq0$ is given by \eqref{source_3}, then $(E_1, \tl E_2, \tl E_3)$ is given by \eqref{lossy_expa} with $p_n^m, q_n^m\neq0$ for some $n$ and $m$, and
\begin{equation}
\left\{\begin{array}{l} \gamma_n^m=\mO(1), \;\;\eta_n^m=\mO(1),\;\;
c_n^m=\mO(\rho^{n+1}), \;\; d_n^m=\mO(\rho^{n+1});\\
\tl c_n^m=\mO(\rho^{n+3/2}), \;\;\tl d_n^m=\mO(\rho^{n+3/2}),\;\;\tl\gamma_n^m=\mO(\rho^{-n+1/2})\;\;\tl\eta_n^m=\mO(\rho^{-n+1/2});\\
\alpha_n^m=\mO(1),\;\;\beta_n^m=\mO(1).
\end{array}\right.
\end{equation}
\end{enumerate}
\end{proposition}
\begin{proof}
%-------------------------------------------------------------------------------------------------------end of comment------------------------------------
In the case that no source/sink is present $(\tl J=0)$, the
boundary condition and transmission conditions in \eqref{damp} imply
(R-1) and the following equations.
\[
\mbox{(R-4)}\left\{\begin{array}{l}\displaystyle
\eps_\tau^{-1/2}\left(\tl\gamma_n^mj_n(k_\tau\omega)+\tl c_n^mh_n^{(1)}(k_\tau\omega)\right)=2\rho\left(\gamma_n^mj_n(2\omg\rho)+c_n^mh_n^{(1)}(2\omg\rho)\right),\\
\eps_\tau^{-1/2}\left(\tl\eta_n^m\mathcal{J}_n(k_\tau\omega)+\tl
d_n^m\mathcal{H}_n(k_\tau\omega)\right)=\eta_n^m\mathcal{J}_n(2\omega\rho)+d_n^m\mathcal{H}_n(2\omega\rho).
\end{array}\right.\]
\[\mbox{(R-5)}\displaystyle\left\{\begin{array}{l}
\mu_\tau^{-1/2}\left(\tl c_n^m\mathcal{H}_n(k_\tau\omega)+\tl\gamma_n^m\mathcal{J}_n(k_\tau\omega)\right)=k_\tau \left(c_n^m\mathcal{H}_n(2\omega\rho)+\gamma_n^m\mathcal{J}_n(2\omega\rho)\right),\\
\mu_\tau^{-1/2}k_\tau\left(\tl
d_n^mh_n^{(1)}(k_\tau\omega)+\tl\eta_n^mj_n(k_\tau\omega)\right)=2\rho\left(d_n^mh_n^{(1)}(2\omega\rho)+\eta_n^mj_n(2\omega\rho)\right).
\end{array}\right.\]
\[\mbox{(R-6)}\displaystyle\left\{\begin{array}{l}
\eps_0^{-1/2}\alpha_n^mj_n(\frac{k\omega}{2})=\eps_\tau^{-1/2}\left(\tl\gamma_n^mj_n(\frac{k_\tau\omega}{2})+\tl c_n^mh_n^{(1)}(\frac{k_\tau\omega}{2})\right),\\
\eps_0^{-1/2}\beta_n^m\mathcal{J}_n(\frac{k\omega}{2})=\eps_\tau^{-1/2}\left(\tl\eta_n^m\mathcal{J}_n(\frac{k_\tau\omega}{2})+\tl
d_n^m\mathcal{H}_n(\frac{k_\tau\omega}{2})\right).
\end{array}\right.\]
\[\mbox{(R-7)}\displaystyle\left\{\begin{array}{l}
\mu_0^{-1/2}k_\tau\alpha_n^m\mathcal{J}_n(\frac{k\omega}{2})=\mu_\tau^{-1/2}k\left(\tl c_n^m\mathcal{H}_n(\frac{k_\tau\omega}{2})+\tl\gamma_n^m\mathcal{J}_n(\frac{k_\tau\omega}{2})\right),\\
\mu_0^{-1/2}k\beta_n^mj_n(\frac{k\omega}{2})=k_\tau\mu_\tau^{-1/2}\left(\tl
d_n^mh_n^{(1)}(\frac{k_\tau\omega}{2})+\tl\eta_n^mj_n(\frac{k_\tau\omega}{2})\right).
\end{array}\right.\]
Solving (R-6-1) and (R-7-1), we obtain
\[\tl c_n^m=l_1\alpha_n^m,\;\; \tl \gamma_n^m=l_2\alpha_n^m,\]
where as $\rho\rightarrow0^+$
\[\begin{split}l_1&=\displaystyle\frac{\eps_\tau^{1/2}\eps_0^{-1/2}\left(j_n(\frac{k\omega}{2})\mathcal{J}_n(\frac{k_\tau\omega}{2})-
\mu_\tau\mu_0^{-1}\mathcal{J}_n(\frac{k\omega}{2})j_n(\frac{k_\tau\omega}{2})\right)}{h_n^{(1)}
(\frac{k_\tau\omega}{2})\mathcal{J}_n(\frac{k_\tau\omega}{2})-\mathcal{H}_n(\frac{k_\tau\omega}{2})j_n(\frac{k_\tau\omega}{2})}=\mO(\rho^{n+3/2}),\\
l_2&=\displaystyle\frac{\eps_\tau^{1/2}\eps_0^{-1/2}\left(\mu_\tau\mu_0^{-1}\mathcal{J}_n(\frac{k\omega}{2})h^{(1)}_n(\frac{k_\tau\omega}{2})
-j_n(\frac{k\omega}{2})\mathcal{H}_n(\frac{k_\tau\omega}{2})\right)}{h_n^{(1)}(\frac{k_\tau\omega}{2})\mathcal{J}_n(\frac{k_\tau\omega}{2})-
\mathcal{H}_n(\frac{k_\tau\omega}{2})j_n(\frac{k_\tau\omega}{2})}=\mO(\rho^{-n+1/2}).\end{split}\]
Plugging the above quantities into (R-4-1) and (R-5-1), we further
have
\[\begin{split}
\mbox{(R-4-1)}\;\;\;& -r_1\alpha_n^m+2\rho\eps_\tau^{1/2} h_n^{(1)}(2\omega\rho)c_n^m=-2\rho\eps_\tau^{1/2} j_n(2\omega\rho)\gamma_n^m,\\
\mbox{(R-5-1)}\;\;\;&
-r_2\alpha_n^m+\mu_\tau^{1/2}k_\tau\mathcal{H}_n(2\omega\rho)c_n^m=-\mu_\tau^{1/2}k_\tau\mathcal{J}_n(2\omega\rho)\gamma_n^m,\end{split}\]
where
\[\begin{split}r_1&=l_1h_n^{(1)}(k_\tau\omega)+l_2j_n(k_\tau\omega)=\mO(\rho^{1/2}),\\
r_2&=l_1\mathcal{H}_n(k_\tau\omega)+l_2\mathcal{J}_n(k_\tau\omega)=\mO(\rho^{1/2}).
\end{split}\]
Then
\[c_n^m=s_1\gamma_n^m,\;\; \alpha_n^m=s_2\gamma_n^m\]
where
\[\begin{split}s_1&=\frac{\mu_\tau\mathcal{J}_n(2\omega\rho)r_1-2\rho j_n(2\omega\rho)r_2}{2\rho h_n^{(1)}(2\omega\rho)r_2-\mu_\tau\mathcal{H}_n(2\omega\rho)r_1}
%=\frac{-j_n(2\omega\rho)r_2+\mathcal{J}_n(2\omega\rho)r_1}{h_n^{(1)}(2\omega\rho)r_2-\mathcal{H}_n(2\omega\rho)r_1}
=\mO(\rho^{2n+1}),\\
s_2&=\frac{2\rho \mu_\tau\eps_\tau^{1/2}\left(-j_n(2\omega\rho)\mathcal{H}_n(2\omega\rho)+
\mathcal{J}_n(2\omega\rho)h_n^{(1)}(2\omega\rho)\right)}{2\rho h_n^{(1)}(2\omega\rho) r_2-\mu_\tau\mathcal{H}_n(2\omega\rho)r_1}
%&=\frac{\mu_\tau\eps_\tau^{1/2}(-j_n(2\omega\rho)\mathcal{H}_n(2\omega\rho)+\mathcal{J}_n(2\omega\rho)h_n^{(1)}(2\omega\rho))}{h_n^{(1)}(2\omega\rho)
%r_2-\mathcal{H}_n(2\omega\rho)r_1}
=\mO(\rho^{n+1}).\end{split}\]
%where one can show
%\[s_1=\mO(\rho^{2n+1}), \;\;s_2=\mO(\rho^{n+1}),\;\; s_3=\mO(\rho^{2n+1}),\;\; s_4=\mO(\rho^{n+1}).\]
By (R-1) as
$\rho\rightarrow 0^+$, we have
\begin{equation}\gamma_n^m=\displaystyle\frac{f_{nm}^{(1)}}{s_1h_n^{(1)}(2\omega)+j_n(2\omega)}=\mO(1),
\end{equation}
which in turn implies
\[c_n^m=\mO(\rho^{2n+1})%=\mO(k^{n+1}\rho^{2n+1})
,\;\; \alpha_n^m=\mO(\rho^{n+1})%=\mO(k^{-n}\rho^{n+1})
,\]
\[\tl c_n^m=\mO(\rho^{2n+5/2})%=\mO(k^{n+1}\rho^{2n+5/2})
,\;\; \tl \gamma_n^m=\mO(\rho^{3/2})%=\mO(k^{-n}\rho^{3/2})
.\]
Similar calculations suggests the other estimates in \eqref{lossy_est}.\\

Statement (ii) is derived from solving (R-1), (R-4), (R-5) and
\[\mbox{(R-6')}\left\{\begin{array}{l}
\eps_0^{-1/2}\left(j_n(\frac{k\omega}{2})\alpha_n^m+h_n^{(1)}(\frac{k\omega}{2})p_n^m\right)=\eps_\tau^{-1/2}\left(\tl\gamma_n^mj_n(\frac{k_\tau\omega}{2})+\tl c_n^mh_n^{(1)}(\frac{k_\tau\omega}{2})\right),\\
\eps_0^{-1/2}\left(\mathcal{J}_n(\frac{k\omega}{2})\beta_n^m+\mathcal{H}_n(\frac{k\omega}{2})q_n^m\right)
=\eps_\tau^{-1/2}\left(\tl\eta_n^m\mathcal{J}_n(\frac{k_\tau\omega}{2})+\tl
d_n^m\mathcal{H}_n(\frac{k_\tau\omega}{2})\right),
\end{array}\right.\]
\[\mbox{(R-7')}\left\{\begin{array}{l}
\mu_0^{-1/2}k_\tau\left(\alpha_n^m\mathcal{J}_n(\frac{k\omega}{2})+p_n^m\mathcal{H}_n\left(\frac{k\omega}{2}\right)\right)=\mu_\tau^{-1/2}k\left(\tl c_n^m\mathcal{H}_n(\frac{k_\tau\omega}{2})+\tl\gamma_n^m\mathcal{J}_n(\frac{k_\tau\omega}{2})\right),\\
\mu_0^{-1/2}k\left(\beta_n^mj_n(\frac{k\omega}{2})+q_n^mh_n^{(1)}\left(\frac{k\omega}{2}\right)\right)=k_\tau\mu_\tau^{-1/2}\left(\tl
d_n^mh_n^{(1)}(\frac{k_\tau\omega}{2})+\tl\eta_n^mj_n(\frac{k_\tau\omega}{2})\right).\end{array}\right.\]
\end{proof}

Using Proposition \ref{prop_lossy}, all our results in Sections 3
and 4 for the lossless approximate EM cloaking can be shown to hold
equally for the lossy approximate cloaking scheme
(\ref{lossy_coeff}). We remark briefly on this here.

\begin{remark}\label{rem:lossy1} The estimates in (i) of Proposition \ref{prop_lossy} imply that,
without an internal source/sink, the EM fields $(\tl E_3, \tl H_3)$
in the cloaked region $B_{1/2}$ degenerates in order $\mO(\rho^2)$.
Whereas for the EM fields $(\tl E_2, \tl H_2)$ in the lossy layer
$B_1\backslash\overline{B_{1/2}}$, it is easily seen
\[M_{n,k_\tau\omega}^m, \nabla\times M_{n,k_\tau\omega}^m=\mO(\rho^n),\;\;\;\; N_{n,k_\tau\omega}^m, \nabla\times N_{n,k_\tau\omega}^m=\mO(\rho^{-n-1}).\]
Then \eqref{lossy_est} and \eqref{lossy_expa} imply that $(\tl E_2,
\tl H_2)$ degenerate in order $\mO(\rho^2)$ as $\rho$ decays. It
follows that the vanishing Cauchy data appears on the inner surface
$\partial B_1^{-}$. Moreover, the boundary operator on $\partial
B_2$ of the approximate cloaking converges to that of the ideal
cloaking in order $\mO(\rho^3)$.

%{\color{red}Furthermore, the convergence is frequency dependent. In fact, by
%straightforward verification, in low frequency regime ($\omega\ll1$), we have
%\[l_1, l_3=\mO(\omega^{2n+1}\rho^{n+3/2}), \;\;\;\;l_2, l_4=\mO(\rho^{-n+1/2}),\]
%\[r_1, r_3=\mO(\omega^n\rho^{1/2}), \;\;\;\;r_2, r_4=\mO(\omega^n\rho^{1/2}),\]
%%implying
%\[s_1, s_3=\mO(\omega^{2n+1}\rho^{2n+1}).\]
%%where $l_3,
%%l_4, r_3, r_4$, and $s_3, s_4$ are calculated in solving for
%%$d_n^m$, $\eta_n^m$, $\tl d_n^m$, $\tl\eta_n^m$ and $\beta_n^m$.
%Replacing $t_1$ and $t_3$ in \eqref{bdry_b2} by $s_1$ and $s_3$, the coefficients of differences on the boundary measurement for lossy approximate cloaking
%satisfy
%\[g_{nm}^{(1)}\sim\omega^n\rho^{2n+1},\quad g_{nm}^{(2)}\sim\omega^{n-1}\rho^{2n+1}.\] Hence the approximate cloaking for low frequency is satisfactory.
%}
\end{remark}

\begin{remark}\label{rem:lossy2}
With an internal point source/sink of the form (\ref{source_3})
present in the cloaked region, the asymptotic estimates of the
coefficients for the corresponding EM fields are given in (ii). By
straightforward verification, one can show near-invisibility for
the lossy approximate cloaking similar to Proposition
\ref{prop:approximate cloaking 2} in the lossless case. On the other
hand, one can also show that both the EM fields $(\tl E_2, \tl H_2)$
and $(\tl E_3, \tl H_3)$ are $\mO(1)$, and hence they do not
degenerate. Moreover, the Cauchy data on the inner surface $\partial
B_1^-$ does not vanish since by (R-4) and (R-5), the terms
associated to $V_n^m$ of $\nu\times\tl E_2$ and $\nu\times\tl H_2$
are $\mO(1)$. These observations suggest that in the limiting case,
the lossy approximate cloaking converges to the ideal cloaking, and
the cloaked region is completed isolated with the EM fields trapped
inside (see Proposition \ref{prop:source limiting} for similar
observations in the lossless case).

\end{remark}

For the frequency dependence of the performances of the lossy
approximate cloakings, we also have completely similar results to
those in the lossless case, which we would not repeat here ( see our
discussion at the end of Sections 3 and 4).

We conclude this section with two more interesting observations. In
\cite{KOVW}, for the approximate acoustic cloaking by employing a
lossy layer, one needs to require that the damping parameter
$\tau\sim\rho^{-2}$, which is not necessary for our present
approximate EM cloaking. On the other hand, it is shown in
\cite{NgVo} that if $\tau$ is allowed to be $\rho$-dependent, one
could achieve near-invisibility uniformly in frequency. However,
such result does not hold for the approximate EM cloaking.

\section{Numerical experiments}

In this section, we carry out some numerical experiments based on
the discussions and calculations in Sections 3, 4 and 5. First we
introduce an electric plane wave of the form
\begin{equation}\label{eq:planewave}
E=e^{-i\omega x\cdot d}P
\end{equation}
with $d=(1,\theta_d,\phi_d)\in\mathbb{S}^2$, $P\in \C^3$ and $d\cdot
P=0$. In the free space, the EM fields $(E,H):=(e^{-i\omega x\cdot
d}P, -e^{-i\omega x\cdot d}d\times P)$ satisfy Maxwell's equations
\[\nabla\times E=i\omega H\quad \nabla\times H=-i\omega E.\]
The spherical wave functions expansions of the EM-fields $(E,H)$ are
given by
\[\left\{\begin{array}{l}E=\displaystyle\sum_{n=1}^\infty\sum_{m=-n}^na_n^m M_{n,\omega}^m(x)+b_n^m \nabla\times M_{n,\omega}^m(x),\\
H=\displaystyle\frac{1}{i\omega}\sum_{n=1}^\infty\sum_{m=-n}^n\omega^2b_n^mM_{n,
\omega}^m(x)+a_n^m\nabla\times M_{n,
\omega}^m(x),\end{array}\right.\] where
\[a_n^m=\frac{f_{nm}^{(1)}}{j_n(2\omega)},\;\;\;\;\; b_n^m=\frac{2f_{nm}^{(2)}}{\mathcal{J}_n(2\omega)},\]
and {\color{black}
\begin{equation}\label{eq:f_planewave}\begin{array}{l}f_{nm}^{(1)}:=f_{nm}^{(1)}(d, P)=\displaystyle\frac{4\pi}{n(n+1)i^n}\overline{M_{n,{\omega}}^m(2d)}\cdot P,\\
\color{black}f_{nm}^{(2)}:=f_{nm}^{(2)}(d,
P)=\displaystyle\frac{4\pi}{n(n+1)\omega i^{n-1}}\overline{
\nabla\times M_{n,{\omega}}^m(2d)}\cdot P.
\end{array}\end{equation}
By \eqref{eq:spherical_harmonics}, we have
\begin{equation}\label{eq:planewave_coeff}a_n^m=-\frac{4\pi}{\sqrt{n(n+1)}i^n}\overline{V_n^m(d)}\cdot P,\quad
b_n^m=\frac{4\pi}{\sqrt{n(n+1)}\omega i^{n-1}}\overline{U_n^m(d)}\cdot P.\end{equation}}
On the boundary $\partial B_2$, one has
\begin{equation}\label{eqn:plane_bdry}\left(\hat x\times e^{-i\omega x\cdot d}P\right)\big|_{\partial B_2}=\sum_{n=1}^\infty\sum_{m=-n}^n\sqrt{n(n+1)}\left(f_{nm}^{(1)}U_n^m(\hat x)+f_{nm}^{(2)}V_n^m(\hat x)\right).\end{equation}

Figure \ref{Fig:planewave1} demonstrates an electric field by taking
the first 15 modes in the above expansion; that is, $n$ is up to
$N=15$. Throughout all our computations, we shall make use of such
truncation when a spherical wave function expansion is considered.
\begin{figure}[htb!]
\begin{center}
\includegraphics[scale=0.4]{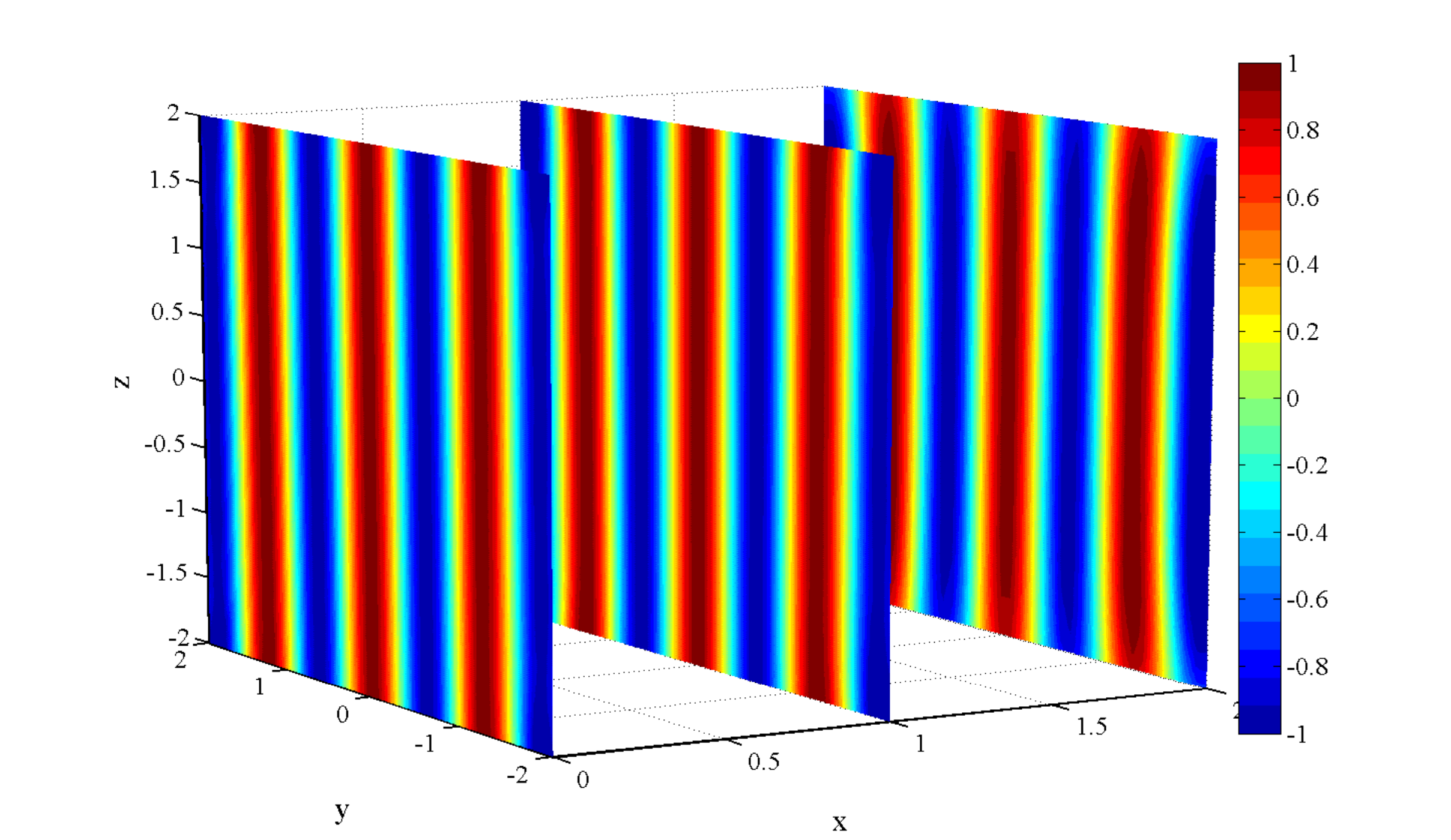}
\end{center}
\caption{\label{Fig:planewave1} Real part of $E_1$, namely the first
component of $E$ (sliced at $x=0,1,2$), with the first 15 modes and
$\omega=5$, $d=(1,\pi/2, \pi/2)\in\mathbb{S}^2$, $P=(1,0,0)^T$.}
\end{figure}
\subsection{Lossless approximate cloaking of passive media}
Recall in Section 3 that the EM material parameters of our lossless
cloaking device are
\[(\tl\eps_\rho(x),\tl\mu_\rho(x))=\left\{\begin{array}{ll}((F_\rho^{(1)})_*I, (F_\rho^{(1)})_*I)\;\;&1<|x|<2,\\ (\eps_0,\mu_0) \mbox{ -- arbitrary constant}&|x|<1.\end{array}\right.\]
Based on the calculations in Lemma \ref{lem: cal_lossless}, we
depict the EM fields $(\tl E_\rho, \tl H_\rho)$ propagating in
$\{B_2;\tl\varepsilon_\rho,\tilde\mu_\rho\}$ in
Figure~\ref{Fig:lossless1} with the following boundary condition
\begin{equation}\label{eq:boundary input}
\hat x\times \tl E_\rho|_{\partial B_2}=\hat x\times E,
\end{equation}
where $E$ is the one demonstrated in Fig~\ref{Fig:planewave1}. It is
remarked that the boundary input (\ref{eq:boundary input}) will also
be implemented in our subsequent numerical experiments, when a
boundary condition is concerned.
\begin{figure}[htb!]
\begin{center}
\includegraphics[scale=0.4]{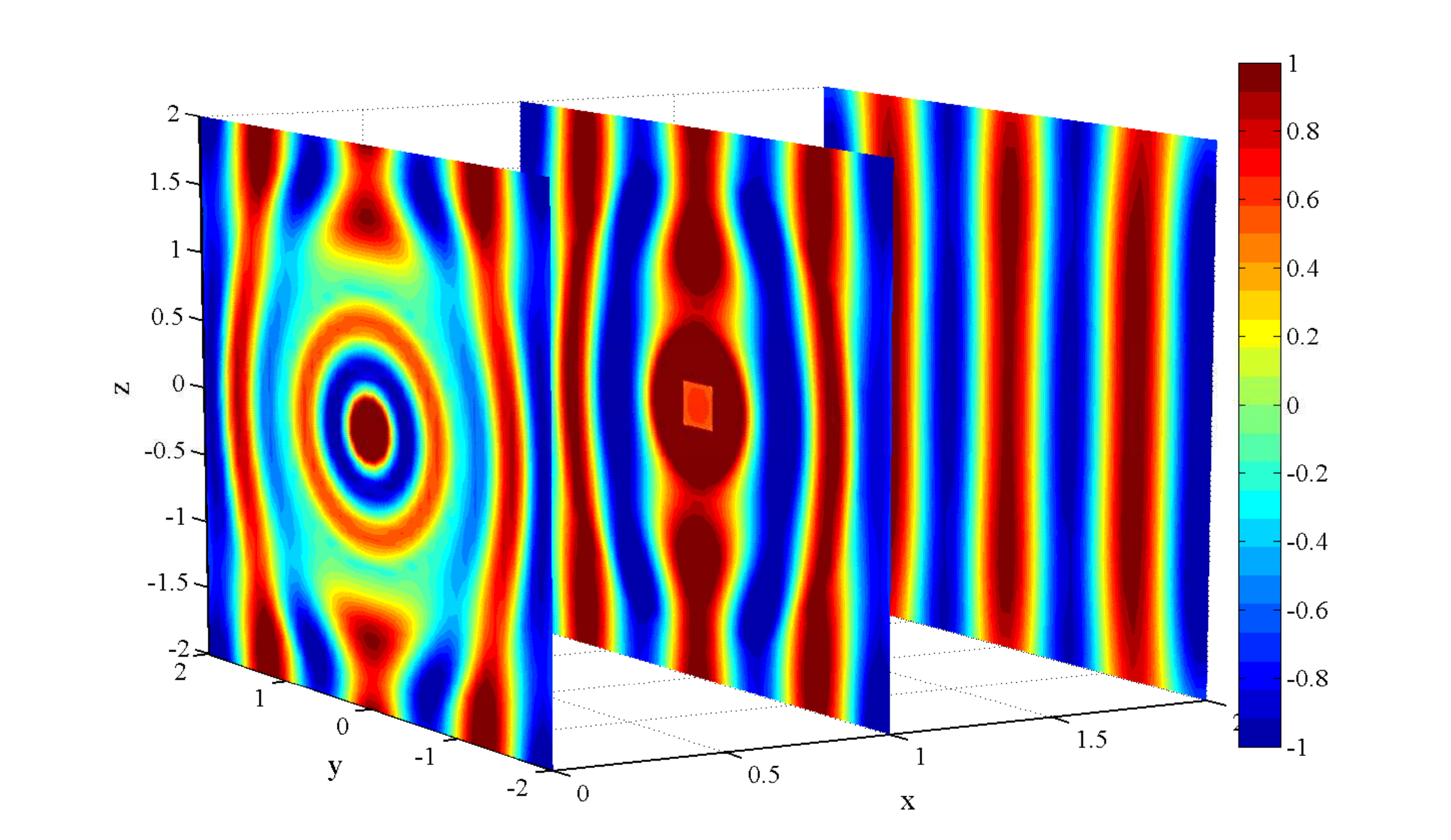}
\end{center}
\caption{\label{Fig:lossless1} Real part of $({\tl{E}}_\rho)_1$
(sliced at $x=0,1,2$), with $\omega=5$, $\eps_0=\mu_0=2$,
$\rho=1/6$. }
\end{figure}
Next, we consider the convergence of the near-cloak to the
ideal-cloak. To that end, for the EM fields $(\tilde{E}_\rho,
\tilde{H}_\rho)$, we compute the deviations of the boundary
operators via the formula
\[Er(\rho):=\|\hat x\times \tl H_\rho-\hat x\times H\|_{H^{-\frac{1}{2}}(\mbox{Div};\partial B_2)}.\]
In our calculations, we shall make use of the following identity from
\cite{Mon},
\[\|\mathbf\lambda\|^2_{H^{-\frac{1}{2}}(\mbox{Div};\partial B_2)}=\sum_{n=1}^\infty\sum_{m=-n}^n\sqrt{n(n+1)}|g^{(1)}_{n,m}|^2+\frac{1}{\sqrt{n(n+1)}}|g^{(2)}_{n,m}|^2,\]
given the vector spherical harmonic expansion of $\mathbf\lambda$
\[\mathbf\lambda=\sum_{n=1}^\infty\sum_{m=-n}^n g^{(1)}_{n,m}U_n^m+g^{(2)}_{n,m}V_n^m.\]
The convergence rate as $\rho\rightarrow 0^+$ is calculated as
following
\[r(\rho):=\ln{\frac{Er(\rho_1)}{Er(\rho_2)}}\Big/\ln{\frac{\rho_1}{\rho_2}},\quad\rho_1,\rho_2\rightarrow 0^+.\]
In Table \ref{Tab:lossless}, we list the computational results,
which verify Proposition \ref{prop:approximate cloaking}, i.e., the
convergence order is $3$.
%{\color{red} Furthermore, we compute the $H^{-\frac{1}{2}}(\mbox{Div};\partial B_2)-$norm error $Er(\rho)$ between boundary measurements for the cloaking medium and the vacuum space
%\[Er(\rho):=\|\hat x\times \tl H_\rho-\hat x\times H\|_{H^{-\frac{1}{2}}(\mbox{Div};\partial B_2)}.\]
%The formula we use is (9.58) in \cite{Mon}
%\[\|\mathbf\lambda\|^2_{H^{-\frac{1}{2}}(\mbox{Div};\partial B_2)}=\sum_{n=1}^\infty\sum_{m=-n}^n\sqrt{n(n+1)}|g^{(1)}_{n,m}|^2+\frac{1}{\sqrt{n(n+1)}}|g^{(2)}_{n,m}|^2,\]
%given the vector spherical harmonic expansion of $\mathbf\lambda$
%\[\mathbf\lambda=\sum_{n=1}^\infty\sum_{m=-n}^n g^{(1)}_{n,m}U_n^m+g^{(2)}_{n,m}V_n^m.\]}
%
%To measure the convergence rate as $\rho\rightarrow0$, we use formula
%\[r(\rho)=\ln{\frac{Er(\rho_1)}{Er(\rho_2)}}\Big/\ln{\frac{\rho_1}{\rho_2}}.\]
%Table \ref{Tab:lossless} verifies Proposition \ref{prop:approximate cloaking}, i.e., the convergence order is $3$.
\begin{table}[htpc]
\begin{center}
\begin{tabular}{c||cccccccc}
\hline
 $\rho$ & $0.1$ &$0.05$&$0.01$&$0.005$&$0.002$&$0.001$ \tabularnewline
\hline
$Er(\rho)$ & $0.1810$ & $0.0139$ & $8.42e-05$ &$1.02e-06$&$6.42e-07$&$7.97e-08$ \tabularnewline
$r(\rho)$ & & $3.703$ & $3.173$ &$3.044$&$3.020$&$3.009$ \tabularnewline
\hline\end{tabular}
\end{center}
\caption{Convergence rate of boundary operator for the lossless
approximate cloaking with $\omega=5$, $\eps_0=\mu_0=2$.}
\label{Tab:lossless}
\end{table}
\subsection{Lossless cloaking of active/radiating objects}
In this numerical experiment, we study the performance of our lossless
approximate cloaking device when an internal point source/sink is
present at origin, elaborating to the discussion in Section 4. We
apply a delta source
$\tl{J}=\sum_{|\alpha|<K}(\partial_x^\alpha\delta_0(x))\mathbf{v}_\alpha$
by introducing a radiating field
\[E_{\tl{J}}=\displaystyle\sum_{n=1}^K\sum_{m=-n}^np_n^mN_{n, k\omega}^m+q_n^m\nabla\times N_{n, k\omega}^m,\]
with known $p_n^m$ and $q_n^m$, into the electric field $E_\rho^-$
inside the virtue inclusion $B_\rho$. In Figure \ref{Fig:lossless_s}
and Table \ref{Tab:lossless_s}, we choose $p_1^{-1}=p_1^0=p_1^1=5$,
$q_1^{-1}=q_1^0=q_1^1=2$ and $q_n^m=p_n^m=0$ otherwise. From
Figure~\ref{Fig:lossless_s}, we see that one could still achieve
near-invisibility, and the EM fields in the cloaked region is {\it
almost} trapped inside. Table~\ref{Tab:lossless_s} verifies that the
convergence order of the near-cloak is 2, which is consistent with
Proposition~\ref{prop:approximate cloaking 2}.
\begin{figure}[htb!]
\begin{center}
\includegraphics[scale=0.4]{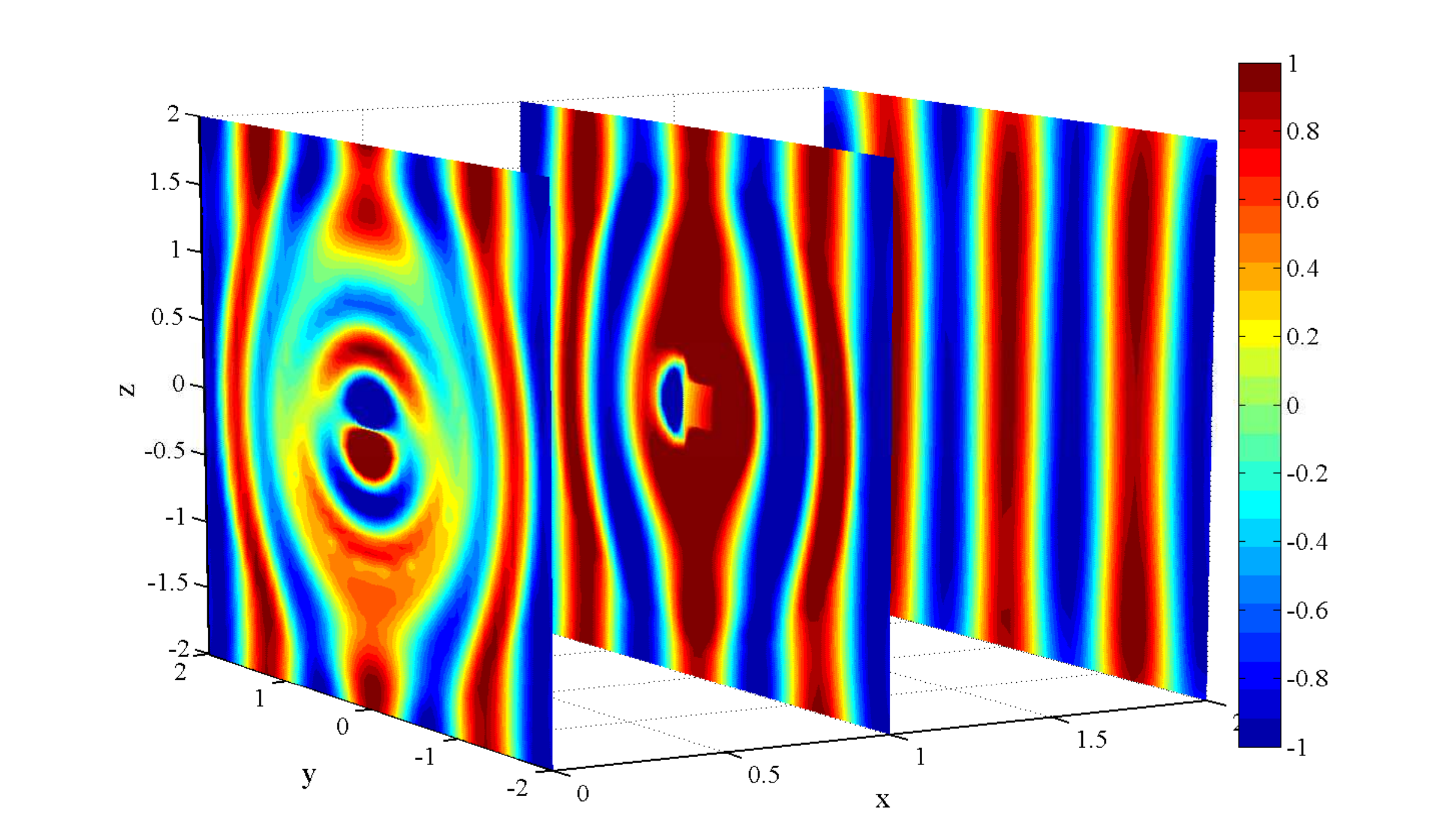}
\end{center}
\caption{\label{Fig:lossless_s} Real part of $({\tl{E}}_\rho)_1$ for
the cloaking problem (sliced at $x=0,1,2$) with a delta source at
the origin, $\omega=5$, $\eps_0=\mu_0=2$, $\rho=1/12$. }
\end{figure}
\begin{table}[htpc]
\begin{center}
\begin{tabular}{c||cccccccc}
\hline
 $\rho$ & $0.1$ &$0.05$&$0.01$&$0.005$&$0.002$&$0.001$ \tabularnewline
\hline
$Er(\rho)$ & $1.9787$ & $0.3509$ & $0.0114$ &$0.0028$&$4.41e-04$&$1.10e-04$ \tabularnewline
$r(\rho)$ & & $2.495$ & $2.129$ &$2.031$&$2.013$&$2.006$ \tabularnewline
\hline\end{tabular}
\end{center}
\caption{Convergence rate of boundary operator for the lossless
approximate cloaking with a delta source, $\omega=5$,
$\eps_0=\mu_0=2$.} \label{Tab:lossless_s}
\end{table}
\subsection{Cloak-busting inclusions and frequency
dependence} In Section 3.3, we have shown the failure of lossless
cloaking due to resonances. In Figures \ref{Fig:resonance1}, for a
fixed $\rho$, the first mode ($n=1$) of boundary errors $Er(\rho)$
are plotted vs frequency $\omega$, for both passive and active
cloaking. We observe blowups of the errors at resonant frequencies,
where the determinants $\mbox{det}(A_{n})$ and $\mbox{det}(B_{n})$
($n=1$) vanish (see Figure \ref{Fig:det_res} for those resonance
frequencies).
\begin{figure}[htb!]
\begin{center}
\includegraphics[scale=0.4]{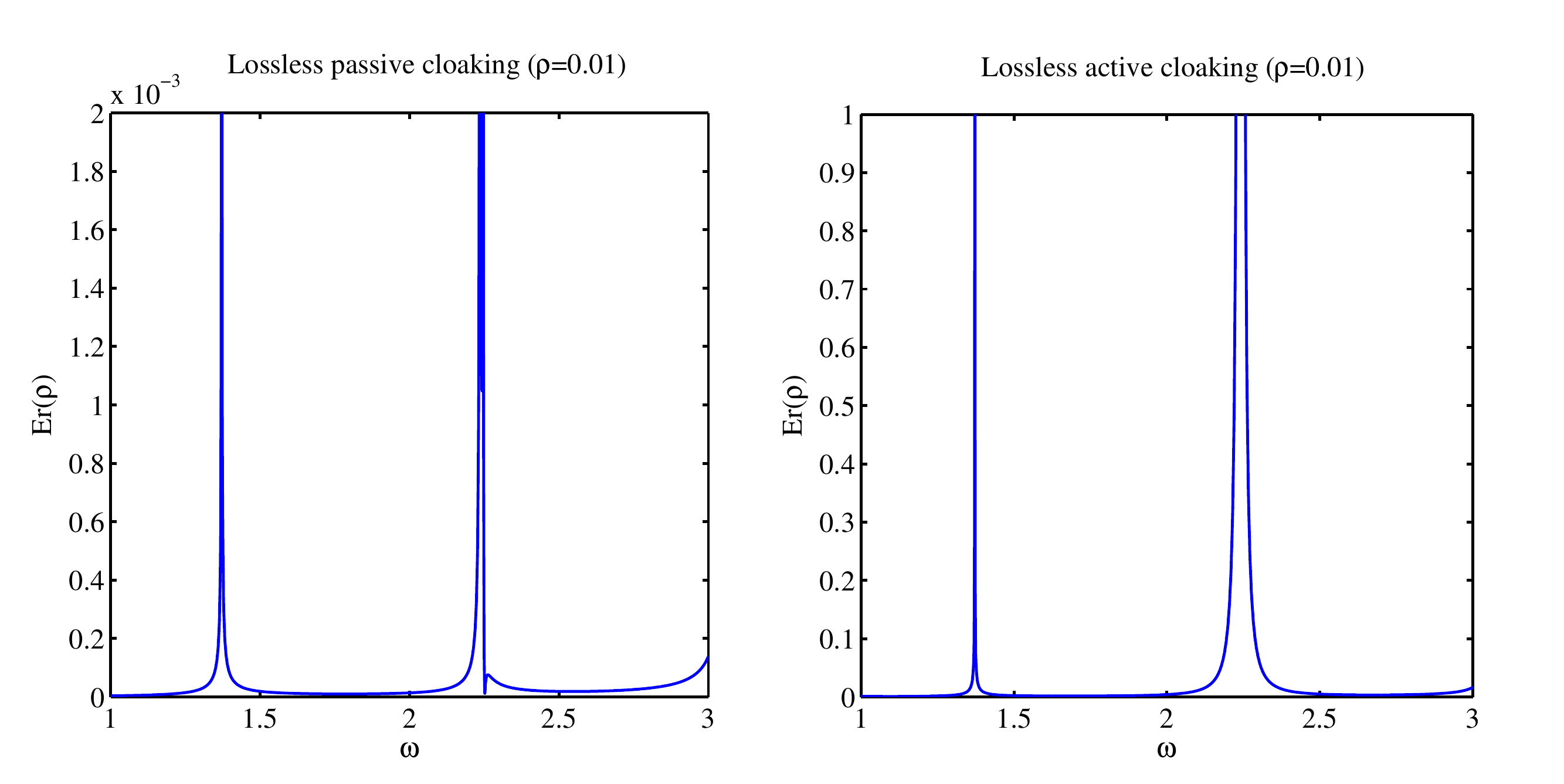}
\end{center}
\caption{\label{Fig:resonance1} Boundary error for mode $n=1$. Left:
lossless cloaking (no source). Right: lossless cloaking (with a
source). $\rho=0.01$, $\omega\in[1,3]$. }
\end{figure}
\begin{figure}[htb!]
\begin{center}
\includegraphics[scale=0.4]{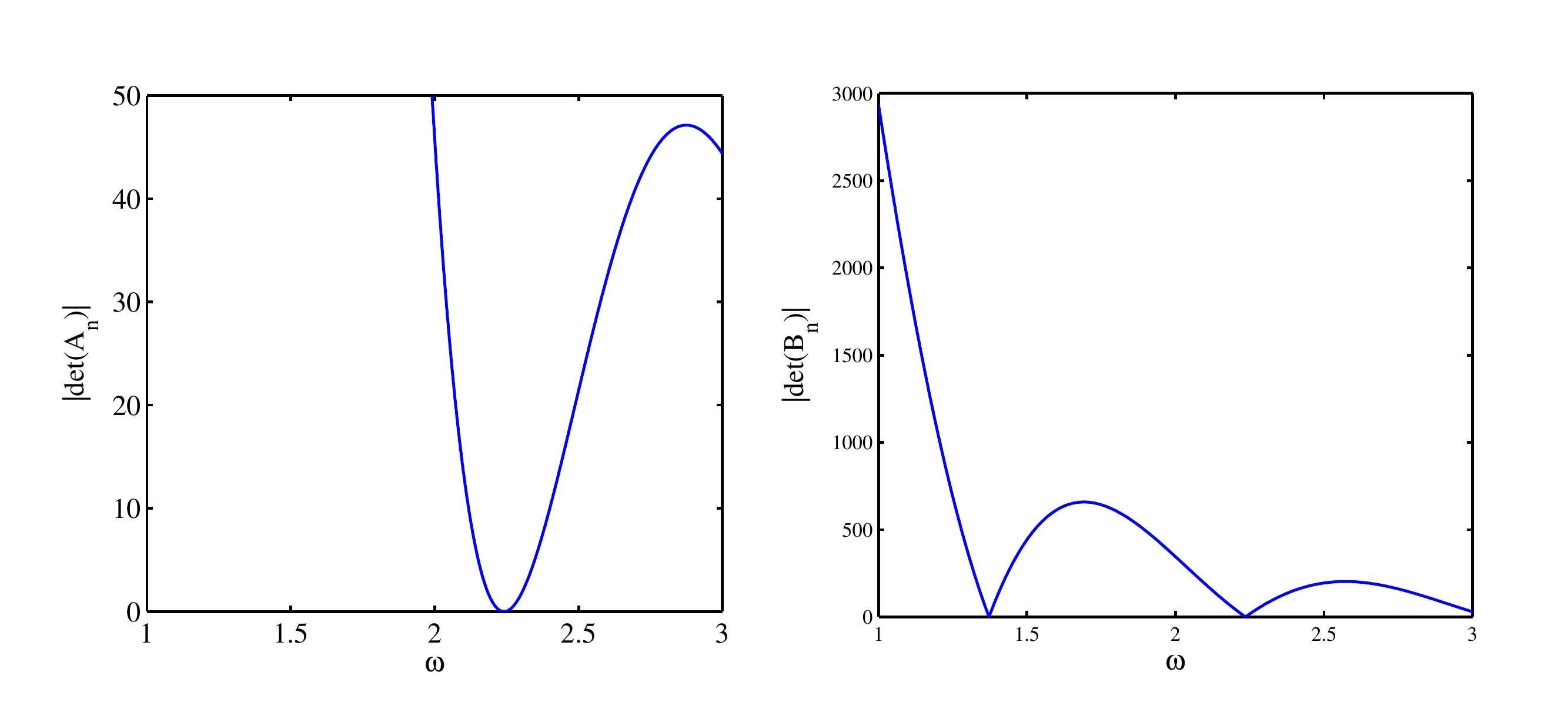}
\end{center}
\caption{\label{Fig:det_res} Frequency dependence of determinants of coefficients system (R-1), (R-2) and (R-3). $\rho=0.01$, $\omega\in[1,3]$.}
\end{figure}
In fact, we have numerically shown that for every frequency $\omega$
and $\rho$, there is a choice of `cloaking-busting' inclusion in
$B_1$, e.g., a pair of parameters $(\eps_0,\mu_0)$ satisfying
\eqref{busting}, such that the lossless construction is resonant. In
Figure \ref{Fig:cloakbust1}, an example of such resonant inclusion
at mode $n=1$ is plotted against $\rho$ for a fixed frequency. One
can see that as $\rho\rightarrow 0^+$, the EM parameters of the
inclusion become singular, namely, $\eps_0\rightarrow\infty$ and
$\mu_0\rightarrow 0$.
\begin{figure}[htb!]
\begin{center}
\includegraphics[scale=0.4]{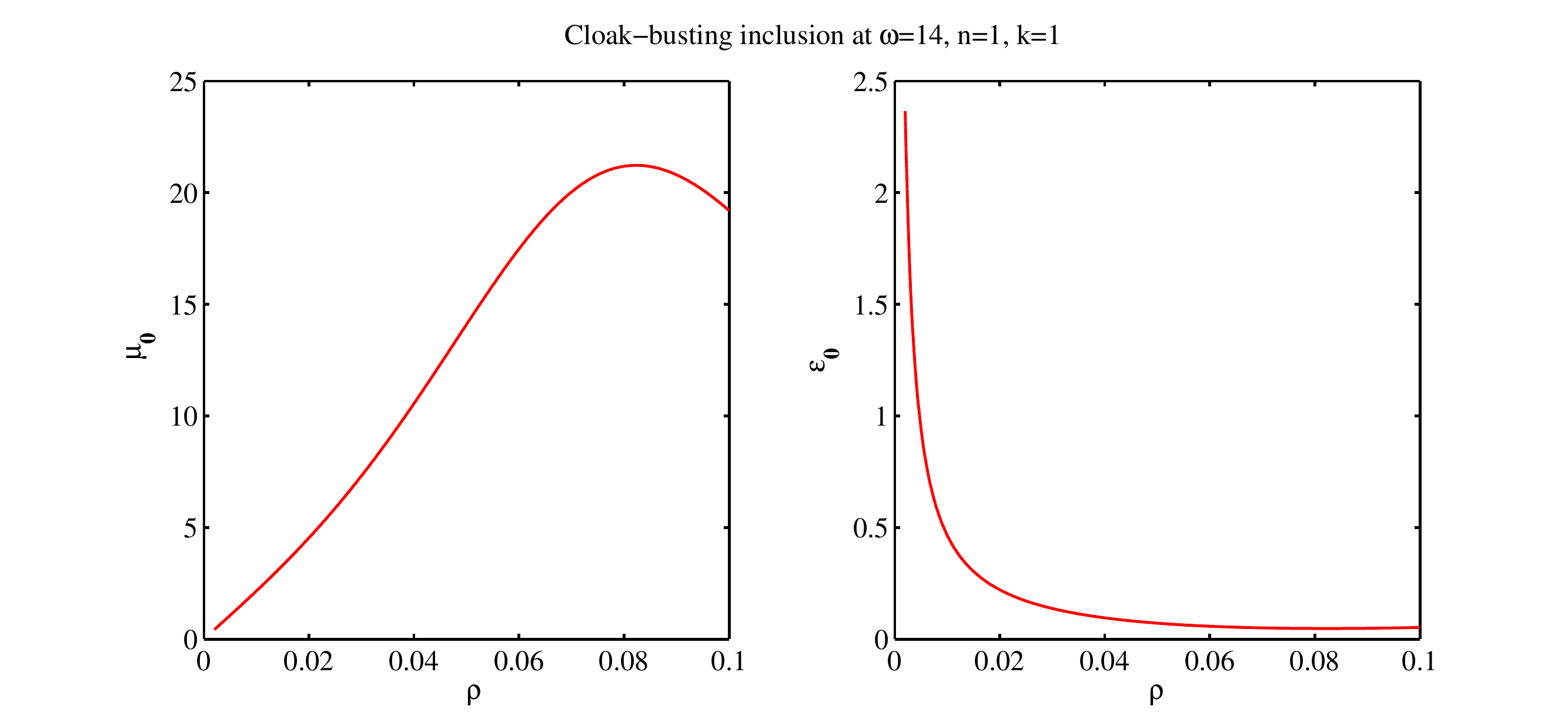}
\end{center}
\caption{\label{Fig:cloakbust1} EM parameters $(\mu_0, \eps_0)$ for
a cloak-busting inclusion at $n=1$ $\omega=14$ with $k=1$. }
\end{figure}
As discussed in Section 3.3 and Section 5, Figure \ref{Fig:low_freq} demonstrates
 that both the lossless (excluding the resonant frequencies) and lossy cloaking schemes
 work well in the low frequency regime, namely when $\omega\ll1$, without any source/sink present in the cloaked region.
 In Sections 4 and 5, when a point source/sink is present at the origin,
 we see that both the lossless and lossy cloaking schemes fail when $\omega\lesssim\rho^{2/3}$,
 as shown in Figure \ref{Fig:low_freq_active}. For higher frequencies, the behaviors of
 the cloaking schemes are not deterministic. Nonetheless, we show in Figure \ref{Fig:high_freq} that the
 lossless cloaking of active/radiating objects (excluding resonant frequencies) generates relatively large
 boundary error $Er(\rho)$ when $\omega\gg 1$.
\begin{figure}[htb!]
\begin{center}
\includegraphics[scale=0.4]{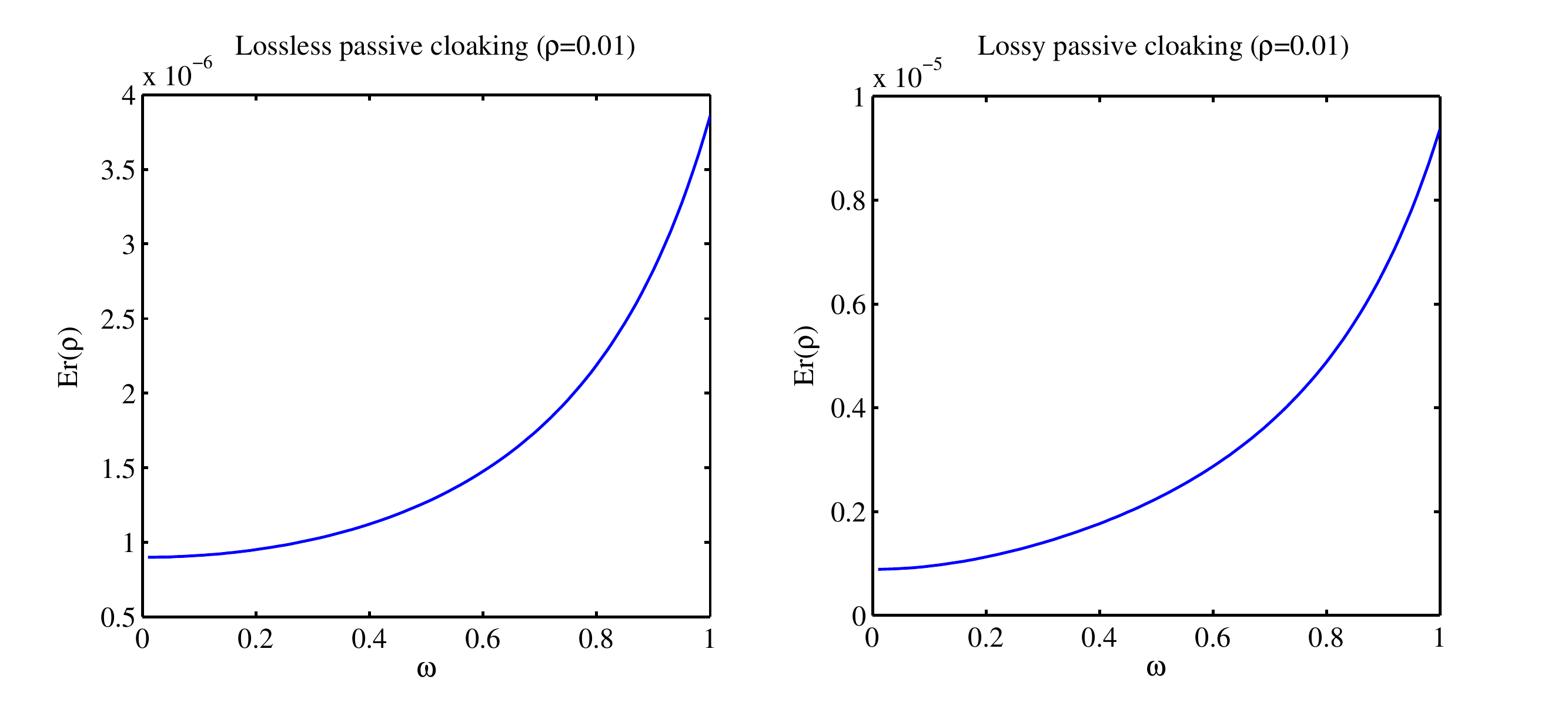}
\end{center}
\caption{\label{Fig:low_freq} Approximate cloaking performance in low frequency regime $\omega\in[0,1]$. Left: boundary error ($n=1$) for lossless cloaking (no source). Right: boundary error ($n=1$) for lossy cloaking (no source). $\rho=0.01$.}
\end{figure}
\begin{figure}[htb!]
\begin{center}
\includegraphics[scale=0.4]{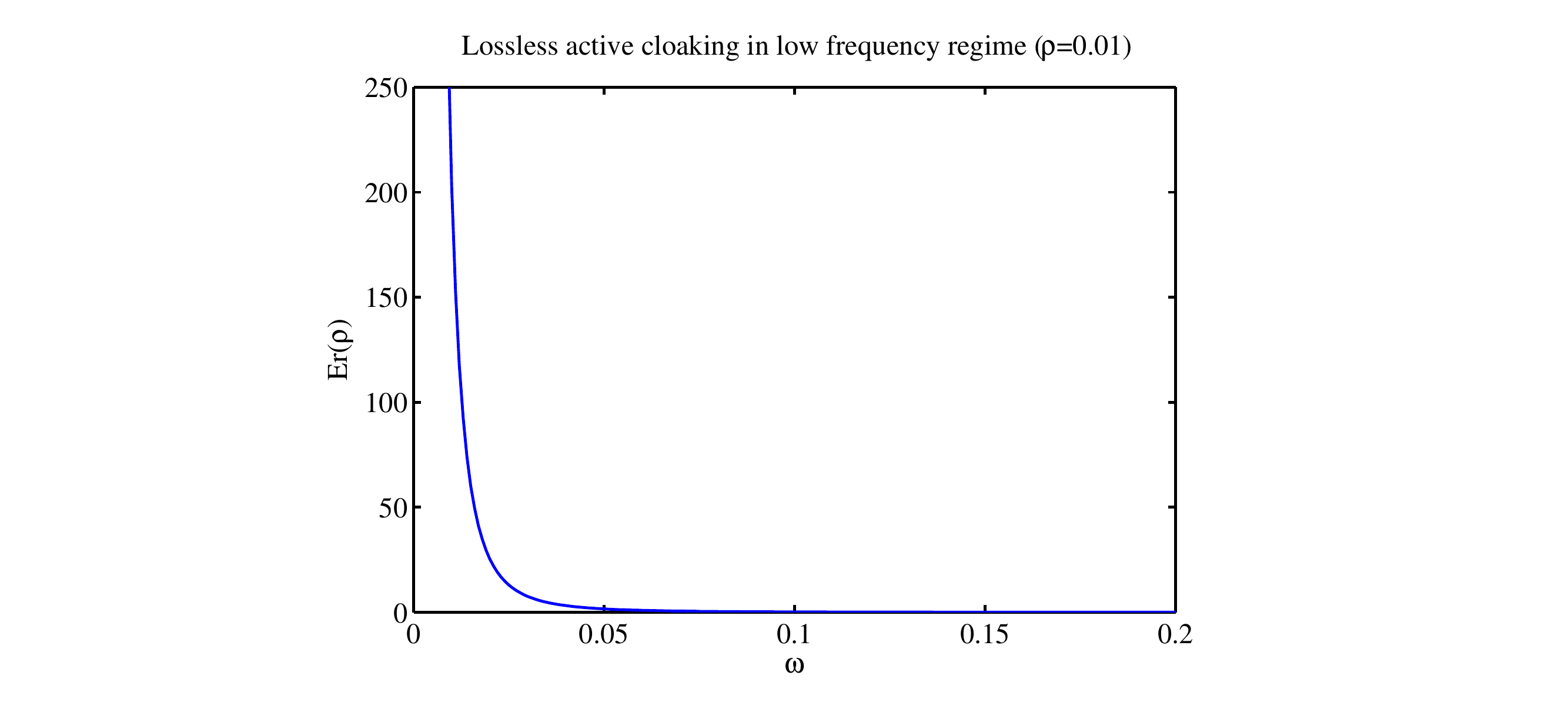}
\end{center}
\caption{\label{Fig:low_freq_active} Boundary error ($n=1$) for cloaking with a source. Approximate cloaking compromises in low frequency regime. $\rho=0.01$. }
\end{figure}
\begin{figure}[htb!]
\begin{center}
\includegraphics[scale=0.4]{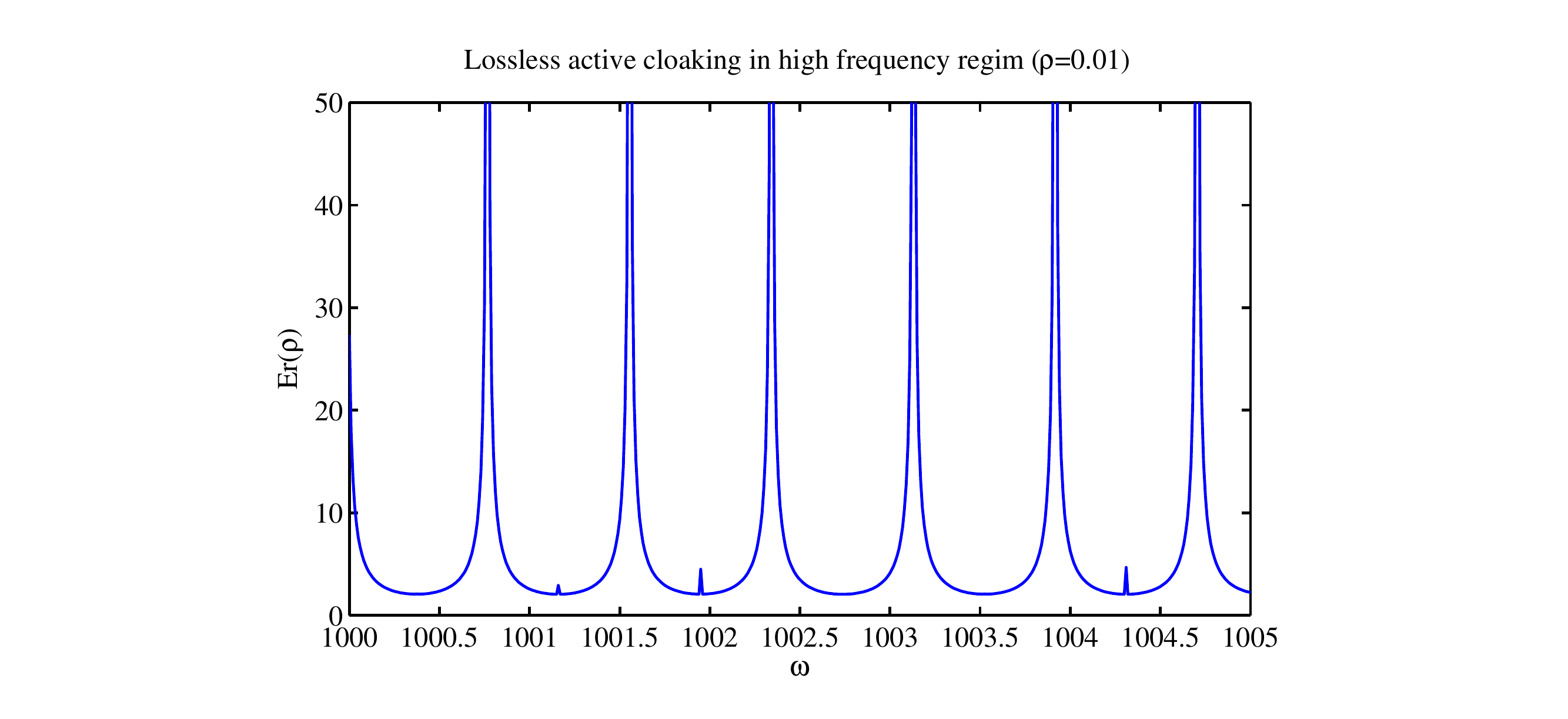}
\end{center}
\caption{\label{Fig:high_freq} Approximate cloaking (with a source) performance in high frequency regime . Boundary error ($n=1$) $Er(\rho)>2$ for $\omega\in[1000,1005]$. $\rho=0.01$.}
\end{figure}
\subsection{Lossy approximate cloaking}
According to our discussion in Section 5, we employ a lossy layer
right between the cloaking layer and the cloaked region. In
Figure~\ref{Fig:lossy}, we show how the EM-fields propagate in such
a lossy construction of approximate cloaking. One can see that
near-invisibility is achieved. In Table \ref{Tab:lossy}, the
convergence order of the lossy near-cloak of passive media as
$\rho\rightarrow 0^+$ is shown to be 3, which is consistent with
Remark~\ref{rem:lossy1}. It is recalled that for the lossy
approximate cloaking, the EM parameters in $B_2$ are given by
\[
(\tl\mu_\rho(x),
\tl\eps_\rho(x))=\left\{\begin{array}{ll}((F_{2\rho})_*I,
(F_{2\rho})_*I)\;\;&1<|x|<2,\\(\mu_\tau,
\eps_\tau):=((F_{2\rho})_*I,
(F_{2\rho})_*(1+i\tau)I)\;\;&\frac{1}{2}<|x|<1,\\ (\mu_0,
\eps_0)&|x|<\frac{1}{2}.\end{array}\right.\]
\begin{figure}[htb!]
\begin{center}
\includegraphics[scale=0.4]{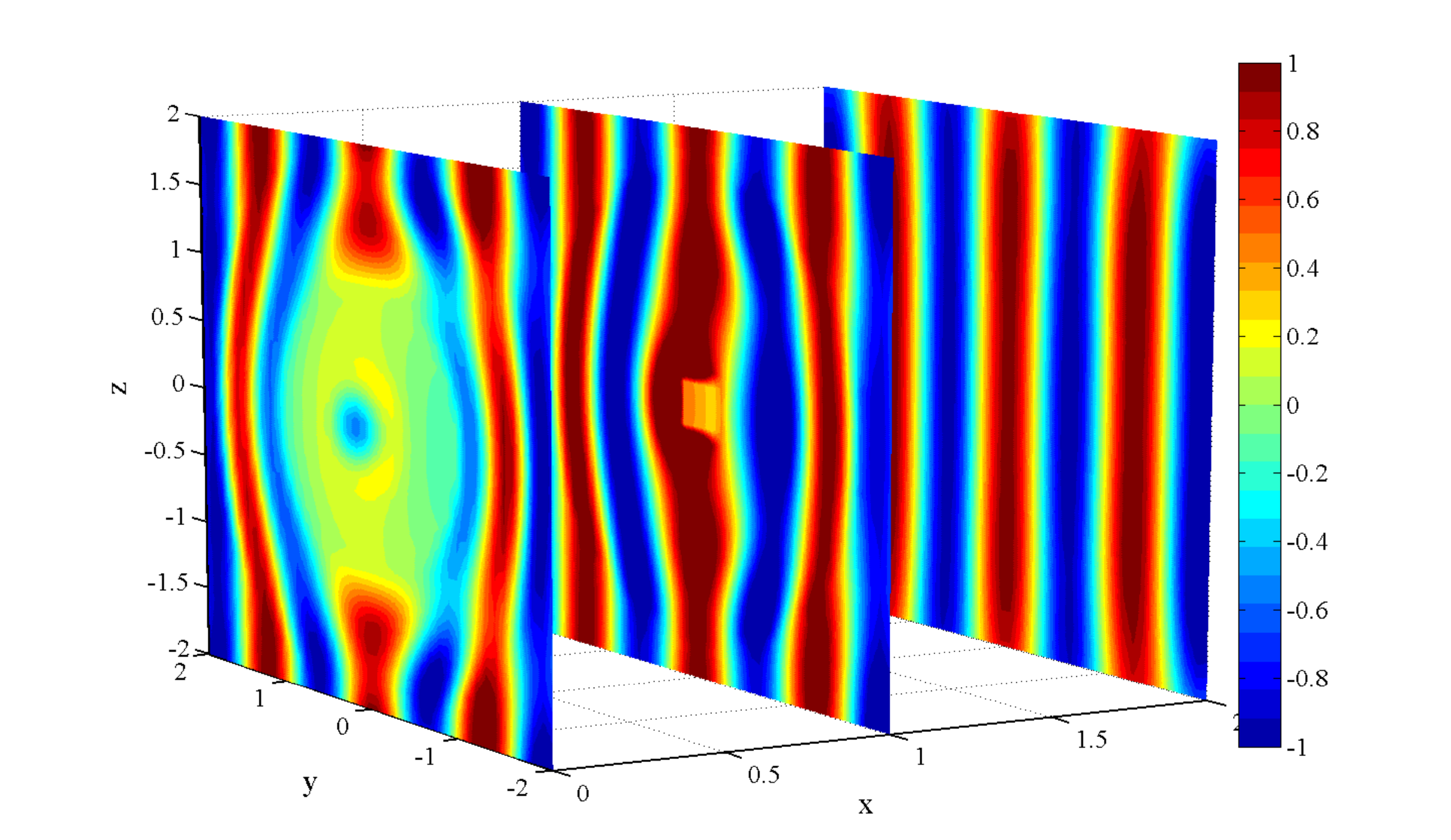}
\end{center}
\caption{\label{Fig:lossy} Real part of $({\tl{E}}_\rho)_1$ for the
lossy approximate cloaking problem (sliced at $x=0,1,2$), with
$\omega=5$, $\eps_0=\mu_0=2$, $\rho=1/6$. }
\end{figure}
\begin{table}[htpc]
\begin{center}
\begin{tabular}{c||cccccccc}
\hline
 $\rho$ & $0.1$ &$0.05$&$0.01$&$0.005$&$0.002$&$0.001$ \tabularnewline
\hline
$Er(\rho)$ & $0.2733$ & $0.0455$ & $3.75e-04$ &$4.69e-05$&$3.00e-06$&$3.75e-07$ \tabularnewline
$r(\rho)$ & & $2.5867$ & $2.9818$ &$2.9998$&$2.9998$&$2.9998$ \tabularnewline
\hline\end{tabular}
\end{center}
\caption{Convergence rate of boundary operator for lossy approximate
cloaking of passive medium, with frequency $\omega=5$,
$\eps_0=\mu_0=2$, damping parameter $\tau=3$.}\label{Tab:lossy}
\end{table}
At last, we demonstrate the frequency dependence of our lossy
approximate cloaking scheme in Figure \ref{Fig:lossy_freq} without a
source/sink. Observe that the resonant frequencies disappear.
However, we observe some frequencies at which the boundary error
$Er(\rho)$ is relatively large. We believe such frequencies are
those very close to the poles or transmission eigenvalues in the complex plane of the boundary
value problem. It is remarked that such phenomenon could also be
observed in the lossless approximate cloaking. If there is a point
source present at the origin, we would have the similar numerical
result as the case considered in Figure~\ref{Fig:high_freq} for the
lossless cloaking.
% Actually, numerically, I did the simpler simulating case where in the inclusion $B(0,1)$, I took $\eps_0=2+1i$ (complex). I observed that at some frequencies,
% the determinants of $A_{n}$ and $B_n$ are very small but not zero. Particularly, $\mathcal{J}_n(\omega\rho)h_n^{(1)}(2\omega)-\mathcal{H}_n(\omega\rho)j_n(2\omega)$ and
%$j_n(\omega\rho)h_n^{(1)}(2\omega)-h_n^{(1)}(\omega\rho)j_n(2\omega)$ are very small, not as the case of resonance, where both are not so small but the determinant is
%zero. This is also can be observed from equation (3.22): at resonant frequencies, (3.22) is satisfied; at current frequencies, LHS of (3.22)- RHS of (3.22) is quite
%large.
\begin{figure}[htb!]
\begin{center}
\includegraphics[scale=0.4]{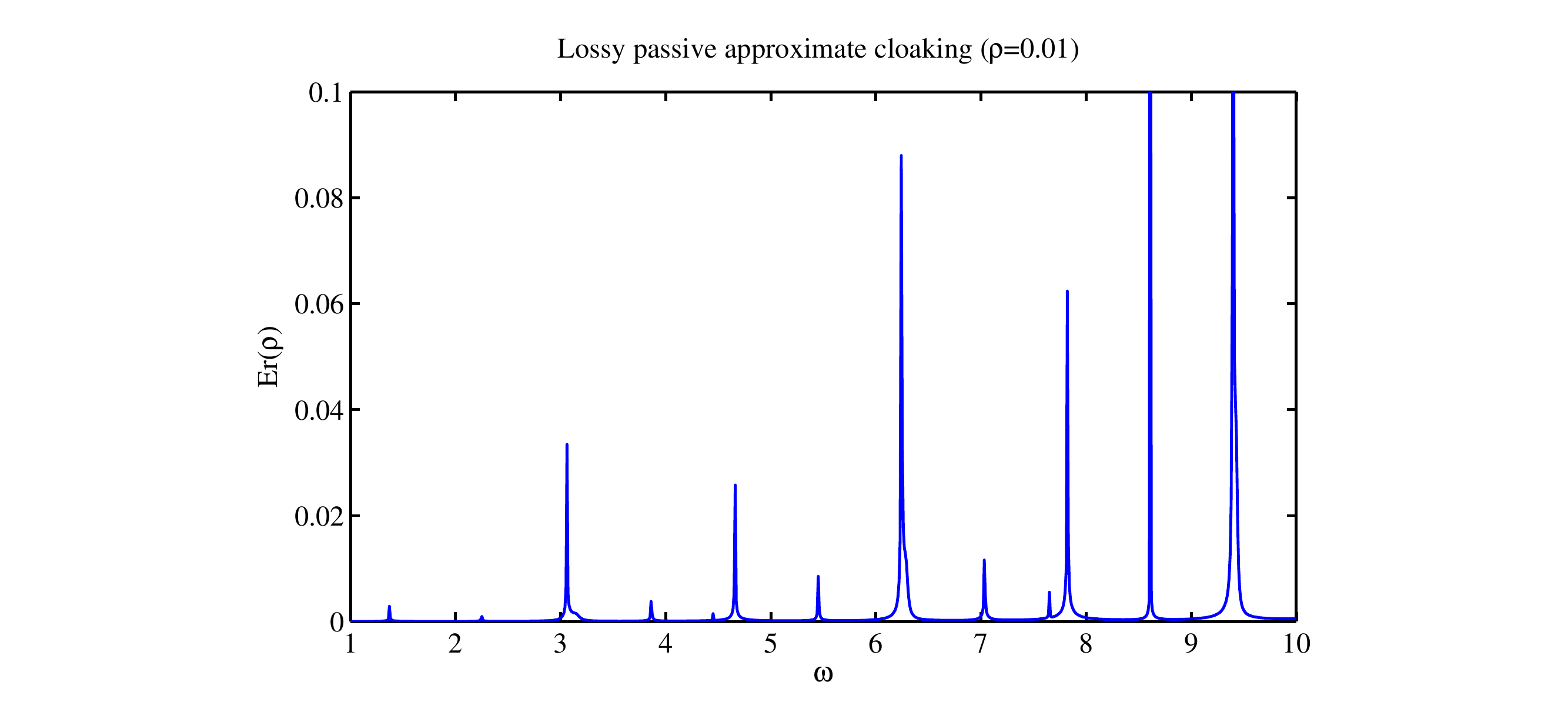}
\end{center}
\caption{\label{Fig:lossy_freq} Boundary error ($n=1$) of lossy approximate cloaking (no source). $\rho=0.01$.}
\end{figure}

\section*{Acknowledgments}

The authors gratefully acknowledge the continuing help from
Professor Gunther Uhlmann. Hongyu Liu was partly supported by NSF grant DMS 0724808. Ting Zhou was partly supported by NSF grants
DMS 0724808 and DMS 0758357.


\begin{thebibliography}{99}

\bibitem{AE} \textsc{Alu, A. and Engheta, N.}, Achieving transparency with plasmonic and
metamaterial coatings, Phys. Rev. E 72 (2005), 016623.


\bibitem{AP}
\newblock \textsc{Calder\'on, A. P.},
\newblock \emph{On an inverse boundary value problem},
\newblock  Seminar
on Numerical Analysis and its Applications to Continuum Physics (R\'io
de Janeiro, 1980), pp. 65--73, Soc. Brasil. Mat., R\'io de Janeiro,
1980.

\bibitem{CK}
\newblock \textsc{Colton, D. and Kress, R.},
\newblock ``Inverse acoustic and electromagnetic scattering theory," 2nd edition
\newblock Springer-Verlag, New York, 1998.


\bibitem{GKLU5} \textsc{Greenleaf, A., Kurylev, Y., Lassas, M., and Uhlmann,
G.}, {Cloaking devices, electromagnetic wormholes and transformation
optics}, SIAM Review, 51 (2009), 3--33.


\bibitem{GKLU}
\newblock \textsc{Greenleaf, A., Kurylev, Y.,  Lassas, M. and
Uhlmann, G.},
\newblock \emph{Full-wave invisibility of active devices at all frequencies},
\newblock Commu. Math. Phys., \textbf{275} (2007), 749--789.

\bibitem{GKLU_2}
\newblock \textsc{Greenleaf, A., Kurylev, Y.,  Lassas, M. and
Uhlmann, G.},
\newblock \emph{Isotropic transformation optics: approximate acoustic and quantum cloaking},
\newblock New J. Phys., \textbf{10} (2008), 115024.

%\bibitem{GKLU_3}
%\newblock A. Greenleaf, Y. Kurylev, M. Lassas and G. Uhlmann,
%\newblock \emph{Invisibility and inverse problems},
%\newblock Bull. AMS, \textbf{46} (2009), 55--97.

\bibitem{GKLU_4}
\newblock \textsc{Greenleaf, A., Kurylev, Y.,  Lassas, M. and
Uhlmann, G.},
\newblock \emph{Improvement of cylindrical cloaking with the SHS lining},
\newblock Opt. Express, \textbf{15} (2007), 12717--12734.

\bibitem{GLU}
\newblock \textsc{Greenleaf, A., Lassas, M. and
Uhlmann, G.},
\newblock \emph{On nonuniqueness for Calder\'on's inverse problem},
\newblock Math. Res. Lett., \textbf{10} (2003), 685.



\bibitem{KOVW}
\newblock \textsc{Kohn, R. V., Onofrei, D., Vogelius, M. S. and
Weinstein, M. I.},
\newblock \emph{Cloaking via change of variables for the Helmholtz equation},
\newblock Comm. Pure Appl. Math., \textbf{63} (2010), 0973--1016..


\bibitem{KSVW}
\newblock \textsc{Kohn, R. V.,  Shen, H., Vogelius, M. S. and
Weinstein, M. I.},
\newblock \emph{Cloaking via change of variables in electric impedance tomography},
\newblock Inverse Problems, \textbf{24} (2008), 015016.

\bibitem{Leo} \textsc{Leonhardt, U.}, Optical conformal mapping,
Science, {312} (2006), 1777--1780.

%\bibitem{LeoPhi} \textsc{Leohardt, U. and Philbin, T. G.},
%General relativity in electrical engineering, New J. Phys.,
%{8}, 247 (2006).

\bibitem{Liu} \textsc{Liu, H. Y.},
Virtual reshaping and invisibility in obstacle scattering, Inverse
Problems, 25 (2009), 045006.


\bibitem{MN} \textsc{Milton, G. W. and Nicorovici, N.-A. P.},
 On the cloaking effects associated with anomalous localized resonance,
Proc. Roy. Soc. A, {462} (2006), 3027--3095.

\bibitem{Mon} \textsc{Monk, P.}, Finite Element Methods for
Maxwell's Equations, Oxford University Press, New York, 2003.


%\bibitem{MO}\textsc{Milton, G. W., Onofrei, D. and Vasquez, F. G.},
%Active exterior cloaking for the 2D Laplace and Helmholtz
%equations, Phys. Rev. Lett., {103} (2009), 073901.


%\bibitem{Moes1999aa} \textsc{Mo{\"e}s, N., Dolbow, J., and Belytschko, T.},
%A finite element method for crack growth without remeshing, Int. J.
%Numer. Meth. Engng., 46 (1999), 131--150.


%\bibitem{Neh} \textsc{Nehari, Z.}, Conformal Mapping, McGraw-Hill, New York,
%1952.

\bibitem{Ngu} \textsc{Nguyen, H. M.}, Cloaking via change of variables for the Helmholtz equation in the whole
space, Commu. Pure Appl. Math., to appear.

\bibitem{NgVo} \textsc{Nguyen, H. M. and Vogelius, M.}, Full range
scattering estimates and their application cloaking, preprint, 2010

\bibitem{Nor}
\textsc{Norris, A. N.}, Acoustic cloaking theory, Proc. R. Soc. A,
464 (2008), 2411-2434.

\bibitem{PenSchSmi} \textsc{Pendry, J. B., Schurig, D., and Smith, D.
R.}, Controlling Electromagnetic Fields, Science, {312} (2006),
1780--1782.



\bibitem{OPS1}%(MR1224101)
\newblock \textsc{Ola, P., P\"{a}iv\"{a}rinta, L. and Somersalo,
E.},
\newblock \emph{An inverse boundary value problem in electrodynamics},
\newblock Duke Math. J., \textbf{70} (1993), no. 3, 617--653.

\bibitem{OS}%(MR1398411)
\newblock \textsc{Ola, P. and Somersalo, E.},
\newblock \emph{Electromagnetic inverse problems and generalized Sommerfeld potentials},
 \newblock SIAM J. Appl. Math., \textbf{56} (1996), no. 4, 1129--1145.


 \bibitem{PSS}
 \newblock \textsc{Pendry, J. B., Schurig, D. and Smith, D. R.},
 \newblock \emph{Controlling electromagnetic fields},
 \newblock Science, \textbf{312} (2006), 1780--1782.

 \bibitem{RYNQ} \textsc{Ruan, Z., Yan, M., Neff, C. W. and Qiu, M.},
 Ideal cylyindrical cloak: Perfect but sensitive to tiny
 perturbations, Phy. Rev. Lett., \textbf{99} (2007), no. 11, 113903.

\bibitem{YYQ}
\textsc{Yan, M., Yan, W., and Qiu, M.}, Invisibility cloaking by
coordinate transformation, Chapter 4 of {\it Progress in
Optics}--Vol. 52, Elsevier, pp. 261--304, 2008.




\end{thebibliography}
\end{document}